\documentclass[reqno]{amsart}            %,draft
\usepackage{graphicx,cite,color}
\usepackage[bookmarks,bookmarksnumbered,bookmarksopen,colorlinks,backref,linkcolor=blue,citecolor=red]{hyperref}%
\usepackage{xcolor}
\usepackage{dsfont}
\usepackage{mathrsfs}
\usepackage{amsmath, amssymb ,amsthm, amsfonts, amsgen}

\newtheorem{theorem}{Theorem}[section]

\newtheorem{lemma}{Lemma}[section]
\newtheorem{corollary}{Corollary}[section]
\newtheorem{remark}{Remark}[section]
\numberwithin{equation}{section}
\newtheorem{proposition}{Proposition}[section]

\begin{document}

\title[Two-phase flow  in  thin cylindric porous media]
 {Muskat-Leverett two-phase flow in thin cylindric porous media: Asymptotic approach}

\author[T. Mel'nyk \& C. Rohde]{ Taras Mel'nyk$^{\flat}$ \ \& \ Christian Rohde$^\natural$}
\address{\hskip-12pt
$^\flat$ Institute of Applied Analysis and Numerical Simulation,
the University of Stuttgart\\
Pfaffenwaldring 57,\ 70569 Stuttgart,  \ Germany
 }
\email{Taras.Melnyk@mathematik.uni-stuttgart.de}

\address{\hskip-12pt  $^\natural$ Institute of Applied Analysis and Numerical Simulation,
the University of Stuttgart\\
Pfaffenwaldring 57,\ 70569 Stuttgart,  \ Germany
}
\email{christian.rohde@mathematik.uni-stuttgart.de }

\begin{abstract}
A reduced-dimensional asymptotic modelling approach is presented for the analysis of two-phase flow in a thin cylinder with  aperture of order $\mathcal{O}(\varepsilon),$ where $\varepsilon$ is a small positive parameter.
  We consider a nonlinear Muskat-Leverett two-phase flow model expressed in terms of a fractional flow formulation and Darcy's law  with a saturation and the reduced pressure as unknown.  We assume that the capillary pressure is non-singular and neglect the acceleration of gravity in  Darcy's law.   Given flows seep through the lateral surface of the cylinder. This exchange process leads to a non-homogeneous Neumann boundary condition with an intensity factor $\varepsilon^\alpha$ $(\alpha \ge 1)$  which controls  the mass transport.
  Furthermore, the absolute permeability tensor comprises an intensity coefficient $\varepsilon^\beta,$ $\beta \in \Bbb R,$ in the transversal directions. The asymptotic behaviour of the solution is studied as $\varepsilon \to 0,$  i.e.
when the thin cylinder shrinks into an interval.
Two qualitatively distinct cases are  discovered in the asymptotic behavior
 of the solution: $\alpha =1 \ \text{and} \ \beta <2,$ and $\alpha> \beta -1 \ \text{and} \ \alpha >1.$
In each of these cases, the asymptotic approximations are constructed for the pressures, saturations and velocities of these flows,  and the corresponding asymptotic estimates are proved in various norms, including  energy and uniform pointwise norms. Depending on the values of the parameters $\alpha$ and $\beta,$ the first terms of the asymptotics are solutions to the corresponding nonlinear elliptic-parabolic system of two differential equations, which is a one-dimensional model of the Muskat-Leverett two-phase flow.
\end{abstract}



\maketitle
\tableofcontents
% ----------------------------------------------------------------

\section{Introduction}\label{Sec:Introduction}

The understanding of the  effective  dynamics  of a multiphasic fluid  in a porous media, which may consist of, for example, gas, various liquids, and sometimes solids,  is of primary importance in a number of fields, including the oil and gas industry, geothermal energy, environmental engineering, soil science and chemical engineering (see e.g. \cite{Bear-1972,Helmig-1997,Sahimi-2023}). It is therefore essential to identify and comprehend the principal parameters and their impact on flow dynamics in porous media, as this has important implications for the optimization of engineering processes.

It is often the case that porous media comprise a multitude of heterogeneous fractures,  which, in turn, exhibit further characteristics that render them as porous media \cite{Adler-Tho-Mour-2013,Bear-2012}. In comparison to their length, the width dimension of a fracture is typically relatively narrow, spanning a range from micrometers to centimeters. In contrast, the length of a fracture can extend to tens of meters.

This article examines a nonlinear Muskat-Leverett two-phase flow model, which describes the flow of two immiscible and incompressible fluids on the Darcy scale, taking into account the effect of capillary pressure at the interfaces between them. We consider this model in a thin cylinder, designated as $\Omega_\varepsilon,$ as an example of a simple fracture.  Here, the small parameter $\varepsilon$ is used to characterise the dimensionless cross-sectional diameter of the cylinder $\Omega_\varepsilon \subset \Bbb R^3$ with the longitudinal variable $x_1 \in (0, \ell)$ and transversal variables  $\overline{x}_1=(x_2, x_3).$

One of the most significant attributes of flow dynamics in porous media is absolute permeability, which is a quantitative characteristic of the permeability of the medium in accordance with the configuration of the solid matrix, or in other words, the pore structure.  In most cases, it is given in the form of a tensor, since the solid matrix depends on the specific direction. In our model, the  absolute permeability tensor is given by the matrix
\begin{equation}\label{permeability}
 \mathbb{K}  =
\left(
\begin{matrix}
 k_1(x_1) & 0 & 0
  \\
  0 & \varepsilon^\beta  k_{22}(x_1, \frac{\overline{x}_1}{\varepsilon}) & \varepsilon^\beta k_{23}(x_1, \frac{\overline{x}_1}{\varepsilon})
  \\[2pt]
  0 &  \varepsilon^\beta k_{32}(x_1, \frac{\overline{x}_1}{\varepsilon}) & \varepsilon^\beta  k_{33}(x_1, \frac{\overline{x}_1}{\varepsilon})
\end{matrix}
\right),
\end{equation}
where $\beta $ is a real number.

After  nondimensionalisation, the governing equations typically exhibit a range of coefficients, both relatively small and large in value. By fixing one parameter and comparing it with other coefficients, it is possible to identify intensity multipliers near certain coefficients.
In the context of our problem, one such intensity factor is represented by the term $\varepsilon^\beta$ in the absolute permeability tensor~ \eqref{permeability}.  The other one, $\varepsilon^\alpha,$ is present in the boundary conditions
$$
 \vec{V}_{w} \boldsymbol{\cdot} {\boldsymbol{\nu}}_\varepsilon = \varepsilon^{\alpha} \, b\, Q(x_1, \tfrac{\overline{x}_1}{\varepsilon}) \quad \text{and} \quad
 \vec{V}_{o} \boldsymbol{\cdot} {\boldsymbol{\nu}}_\varepsilon = \varepsilon^{\alpha} \, \big(1 - b\big)\,
 Q(x_1, \tfrac{\overline{x}_1}{\varepsilon})
$$
on the lateral surface of the cylinder $\Omega_\varepsilon,$  which  connect the flows in the cylinder with the flows in the surrounding environment.
Here $\vec{V}_{w}$ and $\vec{V}_{o}$ are the unknown velocities of two components of the flow, ${\boldsymbol{\nu}}_\varepsilon$ is the outward unit normal to the lateral surface,  $b$ is a fractional flow function of the $w$-phase, and $Q$ is the specified flow rate (injection/extraction) of the mixture.  A rigorous statement of the problem,  written in a fractional flow formulation, i.e. in terms of a saturation and a reduced (global) pressure, as well as the main assumptions are given in Section~\ref{Sec:Statement}, see in particular \eqref{Ell}-\eqref{initial-cond}.

The aim of this work is to study the asymptotic behaviour of the solution as  the small parameter $\varepsilon$ approaches zero, i.e. as the thin cylinder $\Omega_\varepsilon$ shrinks into an interval.
Additionally, the influence of the parameters $\alpha$ and $\beta$ in the intensity factors on fluid dynamics will be identified.
This will facilitate the development of appropriate one-dimensional models that take into account the different regimes in the original problem.

To achieve this goal, we use the asymptotic multiscale approximation methodology, which is a highly effective tool for studying a wide range of perturbed problems. This approach permits the investigation of the impact of additional model characteristics (or parameters) through the incorporation of higher-order terms into the approximation. For this, it is necessary to correctly choose the asymptotic scale, which is affected by various parameters of the problem under consideration.

Two qualitative cases of the asymptotic behaviour of the solution depending on the values of the parameters $\alpha$ and $\beta$ are identified and studied,  namely (see Fig.~\ref{Figure})
$$
 (1) \  \alpha =1 \ \text{and} \ \beta <2,  \qquad (2) \ \alpha> \beta -1 \ \text{and} \ \alpha >1.
 $$

\begin{figure}[htbp]\label{Figure}
\centering
\includegraphics[width=8cm]{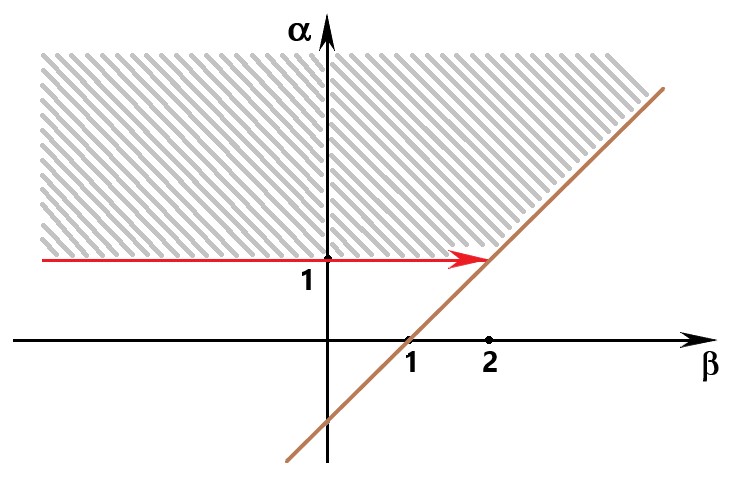}
\caption{The first case corresponds to the  red ray, and the second case to the gray shaded area.}\label{f1}
\end{figure}

In each case, several first terms of the asymptotics are determined, and estimates are provided for the difference between the solution and the corresponding approximation function in various norms, including both energy and uniform  pointwise norms. The first case is considered in Sections~\ref{Sec: 3} and \ref{Sect:4}, while the second one is studied in Section~\ref{Sect:5}. The results obtained in these sections are then used to construct asymptotic approximations for the corresponding pressures and velocities $\vec{V}_{w}$ and $\vec{V}_{o}$
 in Section~\ref{Sect-6}.

It is crucial to highlight the significance of error estimates and convergence rates in substantiating the adequacy of one-dimensional models designed to describe real processes in three-dimensional thin bodies in applied problems.
In our context, these estimates provide compelling evidence of the influence of the intensity factors $\varepsilon^{\alpha}$ and $\varepsilon^{\beta}$. For instance, in the first case, the influence of the parameter $\alpha$ is already observed in the first term of the asymptotics (see the nonlinear limit problem \eqref{limit_prob}, in which the average function of the function $Q$ is present within the differential equations), whereas the impact of the parameter~$\beta$ is discernible in the subsequent term (see e.g. \eqref{Anz-P-be}). In the second case, the influence manifests itself only in the second and the third terms of the asymptotics, respectively (see \eqref{Anz-P-al-be} and \eqref{Anz-S-al-be}).

In this regard, it should be noted that the results of convergence alone used for the analysis of perturbed problems may ultimately prove inadequate and insufficient.
This is due to the fact that, in the limit, a lot of information and detail pertaining to the original problem are lost.
These features can be taken into account only by constructing subsequent terms of the asymptotics.

In the final Section~\ref{Sect: Concl}, we analyze in more detail the results obtained, in particular the role of the parameters $\alpha$ and $\beta.$  As this is the pioneering work to study the influence of such diverse parameters on the two-phase flow dynamics through an asymptotic approach, we have simplified the problem formulation (omitted the acceleration of gravity in the Darcy's law, considered the problem within a rectilinear cylinder and made additional assumptions, particularly regarding the non-singularity of the capillary pressure; see Sect.~\ref{Sec:Statement}). It should be noted here that the Buckley-Leverett model, studied in many works, further simplifies the problem by neglecting capillary pressure and gravitational forces.
Consequently, in Sect.~\ref{Sect: Concl}, these constraints are also discussed and  avenues for future research are presented.

\smallskip

{\bf Literature review.}
The mathematical modelling of the flow of two incompressible and immiscible fluids in porous media has been the subject of intense study over the past several decades. Of particular interest are the existence and uniqueness of solutions of such models due to  permeability degeneracy (see Sect.~\ref{Sec:Statement} for further details).

A number of models for the interaction of single-phase flows in fractured porous media and dimensionally reduced models are described in \cite{Sch-Fle-Hel-2015}. In their  work, the authors put forward a new dimensionally reduced single-phase flow model in fractured porous media, which introduces the simulation of co-dimension one fracture crossings.

The modelling and numerical simulation of two-phase flow in fractured porous media represents a significant area of interest, with a notable trend of ongoing development over recent years \cite{Agh-Dre-Mas-Tre-2021,Bre-Hen-Mas-Sam-2018,Bur-Rohde-2022,Gla-Hel-Fle-Hol-2017,Fum-Sco-2013}.
A fairly complete review on this topic has been presented in \cite{Bur-Rohde-2022}.

Additionally, we cite several publications that employ diverse asymptotic techniques. In \cite{Mor-Show-2012}, convergence results (as $\varepsilon \to 0)$ are obtained for single-phase flow in a porous medium in a region containing a narrow channel with width $\mathcal{O}(\varepsilon)$ and high permeability $\mathcal{O}(\varepsilon^{-1}).$
In examining the convergence results for Richards' equation involved in the fractured porous media model in \cite{Pop-Radu-Kundan-2020}, five distinct reduced problems were identified, depending on the ratio of the
porosities and of the absolute hydraulic conductivities in the fracture and the matrix blocks. In \cite{Dug-Kum-2022}, the authors employ a formal asymptotic approach to present a number of reduced problems related to two-phase flow in a porous medium containing a single fracture. Some of these results have been confirmed by means of numerical examples. The homogenization of two-phase flows in periodic porous media in different settings has been obtained in \cite{Ama-Jur-Pan-Pia-2021,Bou-Luck-Mik-1996,Hen-Ohl-Sch-2013,Jur-Pan-Vrb-2016,Met-Kna-2021,Pannasenko-Vir-2003}
(see also some relevant links therein).

It is regrettable that there are only a few papers  that conduct an asymptotic study of fluid flows in thin fractured-porous media.
Using the method of asymptotic expansions with respect to the thickness of the thin circular cylinder, and assuming that the fluids are strictly separated at $t = 0$ and that only one phase is in contact with the rigid lateral wall, one-dimensional equations are formally identified in \cite{Mikelich-2003} from the classical Navier-Stokes system for two incompressible immiscible flows separated by a free boundary with small surface tension effects.
A single phase flow in a thin a periodic fractured-porous layer with high contrast permeability (the so-called "dual porosity" regime) was considered in \cite{Ama-Pan-Pia-2007}.
The authors model this process using a linear parabolic equation with a rapidly oscillating high-contrast periodic coefficient in an elliptic operator and considering this equation in a thin rectangle (or thin rectangle parallelepiped) with a homogeneous Neumann boundary condition. The corresponding homogenized problems have been derived and convergence results have been proved when the thin rectangle vanishes in the limit, employing the two-scale convergence method.
To the best of our knowledge, no asymptotic studies have been conducted to investigate two-phase flows in thin fractured porous media.

%%%%%%%%%%%%%%%%%%%%%%%

%%%%%%%%%%%
\section{Problem statement}\label{Sec:Statement}
Let $\varpi \subset \mathbb R^2 $ be a bounded simply connected domain with the smooth boundary that contains the origin.  A thin cylinder
 $\Omega_\varepsilon$  is defined as follows
$$
\Omega_\varepsilon =
  \left\{  x=(x_1, x_2, x_3)\in\mathbb{R}^3 \colon \ 0  <x_1 <\ell, \quad  \varepsilon^{-1} \overline{x}_1 \in \varpi
  \right\},
$$
where  $\varepsilon> 0$  is  a small  parameter,  $\overline{x}_1 =  (x_2, x_3).$ We  denote by
$$
\Gamma_\varepsilon := \partial\Omega_\varepsilon \cap \{ x\colon \ 0<x_1<\ell \}
$$
the  lateral surface of  $\Omega_\varepsilon,$ and by
$
\Upsilon_\varepsilon(y_1) := \Omega_\varepsilon \cap \{  x\colon  \ x_1= y_1\}
$
its cross-section at the point $y_1 \in [0, \ell],$ which  is the domain  $\varepsilon\, \varpi $  of diameter $\mathcal{O}(\varepsilon).$

The  thin cylinder $\Omega_\varepsilon$ is  considered as a porous medium with dimensionless porosity $0< const \le \phi(x) < 1$ and the  absolute permeability tensor \eqref{permeability}
that is the symmetric, uniformly  positive-definite  tensor  in $\overline{\Omega}_\varepsilon$ with smooth coefficients of the class $C^2,$ where $\beta \in \Bbb R.$ This means that $k_{23}= k_{32}$ and there are positive constants $C_0, \chi_1$ and $\chi_2$ such that for all $x_1\in [0,\ell]$ and
$\overline{\xi}_1= (\xi_2, \xi_3) \in \varpi ,$  and for all $\overline{\zeta}_1= (\zeta_2,\zeta_3) \in \Bbb R^2,$ we have
\begin{equation}\label{cercitive}
  C_0 \le k_1(x_1), \qquad  \chi_1 |\overline{\zeta}_1|^2 \le \sum_{i,j=2}^{3} k_{ij}(x_1,\overline{\xi}_1) \, \zeta_i \, \zeta_j \le \chi_2 |\overline{\zeta}_1|^2.
\end{equation}
In order to reduce the technical calculations, we will regard that the function  $\phi$   depends only on the variable~$x_1,$
and additionally assume that $\phi \in C^3([0,\ell]).$

 Two immiscible and incompressible fluids (e.g., water and oil  denoted by  $w$ and $o$ respectively) flow through the cylinder. This type of movement is referred to as a two-phase flow.
The key characteristics of such a flow are the saturation $ S_\flat\colon \overline{\Omega}_\varepsilon\times [0,T]\mapsto [0,1]$ and
the velocity $\vec{V}_\flat \colon \overline{\Omega}_\varepsilon\times [0,T]\mapsto \Bbb R^3$ of the $\flat$-th fluid,
where $\flat\in \{o, w\}.$ The saturation~$S_\flat$ refers to the local proportion of the pore space that is occupied by the $\flat$-th phase and
\begin{equation}\label{saturation}
   S_w + S_o =1, \quad S_\flat \in [0, 1].
\end{equation}

Assuming that the density of each phase is constant, the equation of conservation of mass for each phase reads as follows
\begin{equation}\label{conservation}
  \phi \, \partial_t  S_\flat + \nabla \boldsymbol{\cdot} \vec{V}_\flat =0 \quad \text{in} \ \ \Omega_\varepsilon^T := \Omega_\varepsilon\times (0,T).
\end{equation}

The dynamics of the two-phase porous flow process is characterized by Darcy's law (here, the acceleration of gravity is neglected)
\begin{equation}\label{darcy}
  \vec{V}_\flat = -   \lambda_\flat \, \mathbb{K} \nabla P_{\flat} \quad \text{in} \ \ \Omega_\varepsilon^T ,
\end{equation}
where $\nabla = \text{grad}_x,$
$ P_\flat\colon \Omega_\varepsilon^T \mapsto \Bbb R$ is the corresponding phase pressure, $\lambda_\flat := \frac{k_{r \flat}}{\mu_i}$ is the phase mobility, $k_{r \flat}$ is the relative permeability and
the positive constant $\mu_\flat$ is the dynamic viscosity.

The pressures in the two phases are generally not the same, owing to the surface tensions involved. The relationship between them is given by constitutive relation (see, e.g.,  \cite{Leveret})
\begin{equation}\label{capilarity}
 P_o -  P_w  = p_c(S) \quad (S:= S_w),
\end{equation}
where  $p_c\colon  (0,1)\mapsto \Bbb R$ is the capillary pressure. We assume that $p_c \in C^3([0,1])$  and
\begin{equation}\label{Leveret}
  p_c^\prime(S)  < 0 \quad \text{for} \ \ S\in [0,1].
\end{equation}
The assumption of non-singularity of the capillary pressure is a common one in many works in this field (see \cite{Alt-Benedetto-1985,Ama-Jur-Pan-Pia-2021,Atoncev-book-1990,Bou-Luck-Mik-1996,Dug-Kum-2022}).

The functions $k_{r o}$ and $k_{r w}$ are  empirically defined nonnegative  functions of $S:= S_w \in [0,1]$ (see e.g.~\cite{Kro-Luckhaus-1984}) and we assume that they are smooth of class $C^3$ and satisfy
\begin{gather}\label{permeabilities}
   k_{r w}^\prime(S) > 0, \quad k_{r w}(0) =0,
  \\
    k_{r o}^\prime(S) < 0, \quad  k_{r o}(1) =0.
\end{gather}
Then the total mobility  $\lambda := \lambda_o + \lambda_w$ satisfies the inequalities
\begin{equation}\label{total mobillity}
  0 < c_1:= \min_{S\in[0,1]} \lambda(S) \le \lambda(S) \le c_2:= \max_{S\in[0,1]} \lambda(S)   \quad \text{for} \ \ S\in [0,1];
\end{equation}
obviously, $0\le  \lambda_o \le c_2$ and $0\le  \lambda_w \le c_2.$

These properties determine a typical feature of two-phase porous flow. The assumptions made are both physically reasonable and consistent with established theories (see \cite{Alt-Benedetto-1985,Atoncev-book-1990,Bear-1972})
Equations \eqref{saturation} - \eqref{capilarity} represent the Muskat-Leverett two-phase flow model.
Boundary value problems for the system of these equations  are best  posed after certain transformations
 which facilitate the simplification of the system while simultaneously revealing its basic mathematical structure \cite{Atoncev-book-1990,Spivak-1977,Chavent-1986}; the result is the so-called a fractional flow formulation.

Let us briefly recall these transformations. Due to \eqref{saturation} and \eqref{conservation}, the total flow velocity $\overrightarrow{V} = \vec{V}_o + \vec{V}_w$ satisfies the equation $\nabla \boldsymbol{\cdot} \overrightarrow{V} =0,$ and from \eqref{darcy} and \eqref{capilarity} it  follows
\begin{align}\label{total1}
 - \overrightarrow{V} = & \  \lambda_o\, \mathbb{K} \nabla P_{ o} +  \lambda_w\, \mathbb{K} \nabla P_{ w}
   = \lambda_o\, \mathbb{K} \big(\nabla P_{ w}  + p_c^\prime(S) \,\nabla S\big) +  \lambda_w \, \mathbb{K} \nabla P_{ w}
    \\
  = & \  \lambda\,  \mathbb{K} \, \nabla\bigg(P_w +   \int_{0}^{S} \frac{\lambda_o(\eta)}{\lambda(\eta)} \, p^\prime_c(\eta) \,d\eta   \bigg).\notag
   \end{align}
Here, we used \eqref{total mobillity}.
Following \cite{Atoncev-book-1990} and  introducing the "reduced" pressure
\begin{equation}\label{mean Pressure}
  P = P_w +   \int_{0}^{S} \frac{\lambda_o(\eta)}{\lambda(\eta)} \, p_c^\prime(\eta) \,d\eta,
\end{equation}
we get
\begin{equation}\label{total2}
  - \overrightarrow{V} = \lambda\,  \mathbb{K}  \nabla P .
\end{equation}

Now, we express $\vec{V}_w$ in terms of $S$ and $P$:
\begin{align}\label{total3}
  - \vec{V}_w \stackrel{\eqref{darcy}}{=}  & \ \lambda_w \mathbb{K} \nabla P_w \stackrel{\eqref{mean Pressure}}{=} \lambda_w \mathbb{K} \nabla\Big(P -  \int_{0}^{S} \frac{\lambda_o(\eta)}{\lambda(\eta)} \, p_c^\prime(\eta) \,d\eta\Big)
 \\
   = & \  \lambda_w \mathbb{K} \nabla P -  \frac{\lambda_w \, \lambda_o }{\lambda} \, p_c^\prime \,  \mathbb{K} \nabla S
  \stackrel{\eqref{total2}}{=}  \Lambda(S) \,  \mathbb{K} \nabla S - b(S) \overrightarrow{V}, \notag
\end{align}
where $b(S):= \frac{\lambda_w(S)}{\lambda(S)}$ is the fractional flow function (the fluidity) of the $w$-phase  and
\begin{equation}
  \Lambda(S) := -  \frac{\lambda_w(S) \, \lambda_o(S) }{\lambda(S)} \, p_c^\prime(S) \label{Lambda}
   \end{equation}
is the  capillary diffusion coefficient.
 Due to \eqref{Leveret} - \eqref{total mobillity},
 \begin{equation}\label{Lambda1}
   \Lambda(S) >0 \quad \text{for} \ \ S\in (0,1) \quad \text{and} \quad  \Lambda(0) = \Lambda(1) = 0.
 \end{equation}

Obviously, $ \vec{V}_o =  \Lambda(S) \,  \mathbb{K} \nabla S + (1 - b(S)) \overrightarrow{V}.$

The function $S$ satisfies the equation $\phi \, \partial_t  S + \nabla \boldsymbol{\cdot} \vec{V}_w =0$ (see \eqref{conservation}).
Therefore, based on \eqref{total3}, we get
$
\phi \, \partial_t  S =  \nabla \boldsymbol{\cdot} \big(\Lambda(S) \,  \mathbb{K} \nabla S   - b(S)  \overrightarrow{V}\big)  .
$

Since the objective is to derive a one-dimensional model of two-phase flow as the cylinder $\Omega_\varepsilon$ shrinks to the interval
 $\mathcal{I}:= \{x\colon x_1\in (0, \ell), \ \ x_2=x_3=0\}$ for $\varepsilon \to 0$,   we mark the  unknowns  $S, P, \overrightarrow{V}$  and the permeability tensor $\mathbb{K}$ by the subscript $\varepsilon.$
The matrix  $\mathbb{K}_\varepsilon$ has a block-structure (see \eqref{permeability}), and to describe the heterogeneity of the material in the cross-section of  $\Omega_\varepsilon$ at $x_1$ we use the variables $\frac{\overline{x}_1}{\varepsilon}.$ In addition, there is an  intensity factor $\varepsilon^\beta$ in the cross-section block of the permeability tensor.

Thus, we are led to the system of two quasilinear differential equations
 \begin{gather}
  \nabla \boldsymbol{\cdot} \big(\lambda(S_\varepsilon) \,  \mathbb{K}_\varepsilon(x)  \nabla P_\varepsilon \big) = 0
  \quad \text{in} \ \ \Omega_\varepsilon^T, \label{Ell}
 \\ %\end{equation} \begin{equation}
 \phi(x_1) \, \partial_t  S_\varepsilon =  \nabla \boldsymbol{\cdot} \big(\Lambda(S_\varepsilon) \,  \mathbb{K}_\varepsilon(x) \nabla S_\varepsilon  -  b(S_\varepsilon) \, \overrightarrow{V}_\varepsilon \big)  \quad \text{in} \ \ \Omega_\varepsilon^T, \label{Par}
  \end{gather}
 where $ \overrightarrow{V}_\varepsilon = -  \lambda(S_\varepsilon) \,  \mathbb{K}_\varepsilon(x)  \nabla P_\varepsilon .$
 Equation \eqref{Ell} is elliptic with respect to $P_\varepsilon,$ while \eqref{Par} is a degenerate parabolic equation with respect to $S_\varepsilon.$

We supplement systems \eqref{Ell}-\eqref{Par}  with typical boundary conditions that are commonly encountered in various applications:
\begin{gather}\label{bc-cond1}
  S_\varepsilon|_{(x,t)\in \Upsilon^T_\varepsilon(0)} = S^0(0,t) \quad \text{and} \quad  S_\varepsilon|_{(x,t)\in \Upsilon^T_\varepsilon(\ell)} = S^0(\ell,t),
  \\
  \label{bc-cond1+}
  P_\varepsilon|_{(x,t)\in \Upsilon^T_\varepsilon(0)} = q_0(t) \quad \text{and} \quad  P_\varepsilon|_{(x,t)\in \Upsilon^T_\varepsilon(\ell)} = q_\ell(t),
  \\
  \label{bc-cond2}
- \Big( \lambda(S_\varepsilon) \,  \mathbb{K}_\varepsilon(x)  \nabla P_\varepsilon\Big)  \boldsymbol{\cdot} {\boldsymbol{\nu}}_\varepsilon = \varepsilon^{\alpha} Q_\varepsilon(x,t), \quad (x,t) \in \Gamma^T_\varepsilon,
\\ \label{bc-cond3}
- \Big(\Lambda(S_\varepsilon) \,  \mathbb{K}_\varepsilon(x) \nabla S_\varepsilon  -  b(S_\varepsilon) \, \overrightarrow{V}_\varepsilon\Big) \boldsymbol{\cdot} {\boldsymbol{\nu}}_\varepsilon = \varepsilon^{\alpha}\, b(S_\varepsilon) \,
Q_\varepsilon(x,t), \quad (x,t) \in \Gamma^T_\varepsilon,
\end{gather}
where $\Upsilon^T_\varepsilon(0):= \Upsilon_\varepsilon(0) \times (0,T),$ $\Upsilon^T_\varepsilon(\ell):= \Upsilon_\varepsilon(\ell) \times (0,T),$ $\Gamma^T_\varepsilon := \Gamma_\varepsilon \times (0,T),$
${\boldsymbol{\nu}}_\varepsilon$ is the outward unit normal to $\partial \Omega_\varepsilon,$ the parameter $\alpha \ge 1,$ and
$Q_\varepsilon(x,t) := Q(x_1, \tfrac{\overline{x}_1}{\varepsilon}, t).$

A novel aspect of these boundary conditions is the introduction of the intensity factor $\varepsilon^{\alpha}.$
Condition~\eqref{bc-cond2}, which can be rewritten as
$\overrightarrow{V}_\varepsilon \boldsymbol{\cdot} {\boldsymbol{\nu}}_\varepsilon = \varepsilon^{\alpha} Q_\varepsilon$ on $\Gamma^T_\varepsilon,$
means that the rate of flow of the mixture on $\Gamma^T_\varepsilon$  is assumed to be known. Considering \eqref{total3},
condition \eqref{bc-cond3} is reduced to $\vec{V}_{w,\varepsilon} \boldsymbol{\cdot} {\boldsymbol{\nu}}_\varepsilon = \varepsilon^{\alpha} \, b(S_\varepsilon)\, Q_\varepsilon$ and shows that the injection/extraction of the mixture on $\Gamma^T_\varepsilon$ are carried out in proportion to the fluidity of the phases. Since $\overrightarrow{V}_\varepsilon = \vec{V}_{o, \varepsilon} + \vec{V}_{w,\varepsilon},$ we get
 $\vec{V}_{o,\varepsilon} \boldsymbol{\cdot} {\boldsymbol{\nu}}_\varepsilon = \varepsilon^{\alpha} \, \big(1 - b(S_\varepsilon)\big)\, Q_\varepsilon$ on $\Gamma^T_\varepsilon.$

For the saturation $S_\varepsilon(x,t),$  it is also necessary  to specify an initial condition
\begin{equation}\label{initial-cond}
  S_\varepsilon(x,0) = S^0(x_1,0), \quad x \in \Omega_\varepsilon.
\end{equation}

The functions occurring in the right-hand sides of conditions \eqref{bc-cond1} - \eqref{initial-cond} are assumed to be known and to have the following properties:
\begin{itemize}
  \item
  $
  S^0 \in C^{3,2}\big([0, \ell] \times [0, T]\big) \quad \text{and} \quad  0< S^0 < 1 \ \ \text{in} \ \ [0, \ell] \times [0, T];
  $
  \item
  the functions $q_0, q_\ell \in C^2([0,T]),$  $q_0(0)=q_\ell(0)=0,$  $q_0 > 0$  and $q_0(t) > q_\ell(t)$ for $t\in (0,T]$
  (the last inequality indicates that the flow progresses from left to right);
   \item
the function
$
Q(x_1,\overline{\xi}_1,t), \ \ (x_1,\overline{\xi}_1,t) \in \big\{x_1\in [0, \ell], \  \overline{\xi}_1 \in \overline{\varpi}, \   t\in [0,T]\big\},
$
belongs to the class $C^3$ in the  domain of definition,   vanishes  uniformly with respect to $t$  and $\overline{\xi}_1=(\xi_2,\xi_3)$ in $[0, \delta]$ and $[\ell - \delta, \ell],$  where $\delta$ is a fixed small positive number, and $Q|_{t=0}=0.$
\end{itemize}

In order to proceed, it is necessary to define a weak solution to problem \eqref{Ell} - \eqref{initial-cond}, which will be further called problem
$(\mathbb{P}_\varepsilon\mathbb{S}_\varepsilon\!).$ Following \cite[Chapt. 5]{Atoncev-book-1990}, we
say that a pair of  functions $P_\varepsilon$ and $S_\varepsilon$ forms a weak solution to problem $(\mathbb{P}_\varepsilon\mathbb{S}_\varepsilon\!)$ if
\begin{enumerate}
  \item $P_\varepsilon\in L^\infty(\Omega_\varepsilon^T),$ \
  $\nabla P_\varepsilon\in L^\infty\big(0,T; L^{2}(\Omega_\varepsilon)^3\big),$ the  boundary conditions \eqref{bc-cond1+} hold in the sense of traces,
  \item $0 \le S_\varepsilon(x,t) \le 1$ a.e. in $\Omega_\varepsilon^T,$ \
  $\Lambda(S_\varepsilon) \,  \mathbb{K}_\varepsilon \nabla S_\varepsilon \in L^2\big(\Omega_\varepsilon^T\big)^3,$ \ conditions
  \eqref{bc-cond1} and \eqref{initial-cond} are satisfied,
  \item and for any function $\eta$ from the Sobolev space $H^1(\Omega_\varepsilon^T)$ such that $\eta|_{\Upsilon^T_\varepsilon(0)} = \eta|_{\Upsilon^T_\varepsilon(\ell)} =0$ and any function $\zeta$ from the Sobolev space $H^{1}(\Omega_\varepsilon)$ such that $\zeta|_{\Upsilon_\varepsilon(0)} = \zeta|_{\Upsilon_\varepsilon(\ell)} =0,$ we have for a.e. $ t \in (0,T]$
\end{enumerate}
\begin{equation}\label{int1}
  \int_{\Omega_\varepsilon} \overrightarrow{V}_\varepsilon \boldsymbol{\cdot} \nabla\zeta \, dx =  \varepsilon^\alpha \int_{\Gamma_\varepsilon} Q_\varepsilon \, \zeta \, d\sigma_x \qquad \text{and}
\end{equation}
\begin{multline}\label{int2}
   \int_{\Omega_\varepsilon^t} \phi \, S_\varepsilon\, \partial_t \eta \, dxd\tau + \int_{\Omega_\varepsilon^t} \vec{V}_{w,\varepsilon} \boldsymbol{\cdot}\nabla\eta \, dxd\tau - \int_{\Omega_\varepsilon} \phi(x_1) \, S_\varepsilon(x,t) \, \eta(x,t) \, dx \\
  = \varepsilon^\alpha \int_{\Gamma^t_\varepsilon} b(S_\varepsilon) \, Q_\varepsilon \, \eta \, d\sigma_xd\tau  - \int_{\Omega_\varepsilon} \phi(x_1)  \, S^0(x_1,0)\, \eta(x,0)\, dx,
\end{multline}
where $\overrightarrow{V}_\varepsilon$ and $\vec{V}_{w,\varepsilon}$ are defined by \eqref{total2} and \eqref{total3}, respectively.

% \smallskip

The above assumptions are sufficient to demonstrate the existence of a weak solution to problem $(\mathbb{P}_\varepsilon\mathbb{S}_\varepsilon\!)$ (see e.g. \cite{Alt-Benedetto-1985,Atoncev-book-1990,Arbogast-1992,Chavent-1986,Chen-2001,Kro-Luckhaus-1984,Kruz-Suk-1977}).
The issue of the uniqueness of weak solutions is more challenging; it has been established under specific conditions that guarantee some regularity of the weak solution:  for degenerate elliptic-parabolic systems of two-phase flow  in  \cite{Chen-2001}, and for nondegenerate ones in \cite{Atoncev-book-1990,Kruz-Suk-1977}.

In our case, the absence of gravity and the assumptions about the initial saturation $S^0$  and the capillary pressure $p_c$ ensure the
maximum principle (see \cite[Chapt. 5, \S 6]{Atoncev-book-1990})
\begin{equation}\label{saturation1}
     0< \delta_0 := \min_{(x,t)\in \overline{\Omega}_\varepsilon^T} S^0(x_1,t) \le S_\varepsilon(x,t) \le \max_{(x,t)\in \overline{\Omega}_\varepsilon^T} S^0(x_1,t) =: \delta_1  < 1
\end{equation}
for a.e. $(x,t) \in \Omega_\varepsilon^T.$
The left inequality in \eqref{saturation1} means that the parabolic equation \eqref{Par} is regular, i.e.
\begin{equation}\label{regular}
  0< \delta_2 \le \Lambda\big(S_\varepsilon(x,t)\big) \quad \text{for a.e.} \ (x,t) \in \Omega_\varepsilon^T,
\end{equation}
 and  there are no stagnant zones in the flow, where limit values $\{0, 1\}$ of $S_\varepsilon$ are reached. Thus, problem $(\mathbb{P}_\varepsilon\mathbb{S}_\varepsilon\!)$ is regular and has  a unique weak solution.

%\smallskip

Our goal is to construct and justify the  asymptotic approximation for the solution $(P_\varepsilon, S_\varepsilon)$  to problem $(\mathbb{P}_\varepsilon\mathbb{S}_\varepsilon\!)$  as $\varepsilon \to 0,$ i.e., when  the thin cylinder $\Omega_\varepsilon$ shrinks into the interval $ \mathcal{I} :=  \{x\colon  x_1 \in  (0, \ell), \ \overline{x}_1  = (0, 0)\},$ and in addition, to study the impact of the parameters $\alpha$ and $\beta$ on the asymptotic behavior of the solution.

%%%%%%%%%%%%%%%%%%%%

\section{The subcase $\alpha =1$ and $\beta =0$}\label{Sec: 3}

Problem $(\mathbb{P}_\varepsilon\mathbb{S}_\varepsilon\!)$ contains three parameters:  the parameter $\varepsilon,$  the intensity parameter $\alpha$ in the boundary conditions on $\Gamma_\varepsilon,$ and the intensity parameter $\beta$ in  the  absolute permeability tensor $ \mathbb{K}_\varepsilon.$  In such cases, as has been shown in \cite{Mel-Klev-AA-2019,Mel-AnAppl-2021,Mel-Roh_AsAn-2024,Mel-Roh_JMAA-2024}, to construct an approximation it is necessary to adjust the asymptotic scale to those parameters and  the problem under study. In many cases this leads to a new understanding of the asymptotic series \cite{Mel-AnAppl-2021,Mel-Roh_AsAn-2024}.

We start with this subcase of the first case  in order to describe all  stages of the study in detail.
Moreover, the asymptotic estimates are more accurate in this subcase.

\subsection{Formal asymptotic analysis}
From \eqref{saturation1} it follows that  the first term of the asymptotics of $S_\varepsilon$ as $\varepsilon \to 0$ will be a function $s_0\colon \mathcal{I}\times[0,T] \mapsto \Bbb R$  that satisfies the inequalities
\begin{equation}\label{limit-saturation}
     0< \delta_0  \le s_0(x_1,t) \le \delta_1  < 1, \quad \text{for} \ (x_1,t) \in \mathcal{I}^T := \mathcal{I}\times[0,T].
\end{equation}

Using  the inequality
\begin{equation}\label{ineq1}
      \varepsilon \int_{\Gamma_\varepsilon} u^2 \, d\sigma_x \leq C \Bigg( \varepsilon^2
      \int_{\Omega_\varepsilon}|\nabla_{\overline{x}_1} u|^2 \, dx
 +    \int_{\Omega_\varepsilon} u^2 \, dx \Bigg) \quad \forall \, u\in H^{1}(\Omega_\varepsilon),
\end{equation}
which was proved in \cite[\S 2]{M-MMAS-2008}, inequalities \eqref{cercitive} and  \eqref{total mobillity} and the assumptions for the functions $q_0, \, q_\ell$ and $Q,$ we derive from \eqref{Ell}, \eqref{bc-cond1+} and \eqref{bc-cond2}  the  inequality
\begin{equation}\label{app_estimate}
\max_{t\in [0,T]}\tfrac{1}{\sqrt{\upharpoonleft  \Omega_\varepsilon \upharpoonright_3}}\,
{\|P_\varepsilon\|}_{H^{1}(\Omega_\varepsilon)} \le
 {C}_1.
\end{equation}
 Here $\upharpoonleft\!\!S\!\!\upharpoonright_n$  is the $n$-dimensional Lebesgue measure of a set $S.$
\begin{remark}
Here and further, constants in all inequalities are positive and independent of the solution $(P_\varepsilon, S_\varepsilon)$ and the parameters $\varepsilon,$ $\alpha$ and $\beta.$  Mostly constants with the same indices in different inequalities are different.
\end{remark}

Thus, the first term of the asymptotics of $P_\varepsilon$ is of order $\mathcal{O}(1)$ in the norm of the Sobolev space $H^{1}(\Omega_\varepsilon)$ as $\varepsilon \to 0.$ Taking this and \eqref{limit-saturation} into account,  we propose the following ansatzes:
\begin{equation}\label{Anz-P}
 \mathfrak{P}_\varepsilon(x,t) := p_0(x_1,t) + \varepsilon p_1(x_1,t) + \varepsilon^2 u_2\Big(x_1, \dfrac{\overline{x}_1}{\varepsilon}, t \Big)  + \varepsilon^3 u_3\Big(x_1, \dfrac{\overline{x}_1}{\varepsilon}, t \Big)
\end{equation}
and
\begin{equation}\label{Anz-S}
 \mathfrak{S}_\varepsilon(x,t) := s_0(x_1,t) + \varepsilon s_1(x_1,t) + \varepsilon^2 v_2\Big(x_1, \frac{\overline{x}_1}{\varepsilon}, t \Big)   + \varepsilon^3 v_3\Big(x_1, \dfrac{\overline{x}_1}{\varepsilon}, t \Big)
\end{equation}
to construct  asymptotic approximations for $P_\varepsilon$ and $S_\varepsilon,$ respectively.

Substituting $\mathfrak{P}_\varepsilon$ and $\mathfrak{S}_\varepsilon$ in \eqref{Ell} instead of $P_\varepsilon$ and $S_\varepsilon$ respectively, using Taylor’s formula  for $\lambda,$ collecting terms of the same powers of $\varepsilon$ and equating those sums  to zero, we obtain differential equations
 in rescale variables $\overline{\xi}_1 = \frac{\overline{x}_1}{\varepsilon},$ $\overline{\xi}_1 = (\xi_2, \xi_3),$   for the coefficients $u_2, u_3$ and $v_2, v_3$
 in the rescale cross-section $\Upsilon(x_1) := \big\{ \overline{\xi}_1\colon \  \overline{\xi}_1 \in \varpi\big\}.$

Namely, at $\varepsilon^0$ we get
\begin{equation}\label{eq-u-2}
  -\nabla_{\overline{\xi}_1} \boldsymbol{\cdot}\big(\lambda(s_0(x_1,t)) \, \mathbf{K}\nabla_{\overline{\xi}_1} u_2(x_1,\overline{\xi}_1,t)\big)=
  \partial_{x_1}\big(\lambda(s_0(x_1,t)) \, k_1(x_1)\, \partial_{x_1}p_0(x_1,t)\big)
\end{equation}
for $\overline{\xi}_1 \in \Upsilon(x_1),$ where
\begin{equation}\label{matrix-K}
 \mathbf{K}(x_1,\overline{\xi}_1) =
\left(
\begin{matrix}
 k_{22}(x_1, \overline{\xi}_1) &  k_{23}(x_1, \overline{\xi}_1)\\ %[2mm]
 k_{32}(x_1, \overline{\xi}_1) & k_{33}(x_1, \overline{\xi}_1)
\end{matrix}
\right),
\end{equation}
and at $\varepsilon^1$
\begin{equation}\label{eq-u-3}
  -\nabla_{\overline{\xi}_1} \boldsymbol{\cdot}\Big(\lambda(s_0)  \mathbf{K}\nabla_{\overline{\xi}_1}u_3
  + s_1  \lambda'(s_0)   \mathbf{K}\nabla_{\overline{\xi}_1}u_2\Big)
   =
  \partial_{x_1}\Big(\lambda(s_0)  k_1 \partial_{x_1}p_1 + s_1  \lambda'(s_0)  k_1\, \partial_{x_1}p_0 \Big).
\end{equation}

Substituting $\mathfrak{P}_\varepsilon$ and $\mathfrak{S}_\varepsilon$ in \eqref{bc-cond2} and equating  terms with the same order of $\varepsilon,$ we derive
\begin{equation}\label{bc-u-2}
  - \lambda(s_0) \, \big(\mathbf{K}\nabla_{\overline{\xi}_1} u_2\big) \boldsymbol{\cdot} \bar{\nu}(\overline{\xi}_1) = Q(x_1,\overline{\xi}_1,t), \quad \overline{\xi}_1 \in \partial\Upsilon_\varepsilon(x_1),
\end{equation}
\begin{equation}\label{bc-u-3}
  \big( \lambda(s_0) \, \mathbf{K}\nabla_{\overline{\xi}_1} u_3 +
  s_1 \, \lambda'(s_0)  \,  \mathbf{K}\nabla_{\overline{\xi}_1} u_2\big) \boldsymbol{\cdot} \bar{\nu}(\overline{\xi}_1)
  = 0, \quad \overline{\xi}_1 \in \partial\Upsilon_\varepsilon(x_1).
\end{equation}
Here, $\bar{\nu}(\tfrac{\overline{x}_1}{\varepsilon})=\big(\nu_2(\tfrac{\overline{x}_1}{\varepsilon}), \nu_3(\tfrac{\overline{x}_1}{\varepsilon})\big)$
is the outward unit normal   to the boundary of the disk  $\Upsilon_\varepsilon(x_1);$ obviously,
the outward unit normal ${\boldsymbol{\nu}}_\varepsilon$ to the lateral surface of the thin cylinder $\Omega_\varepsilon$ is equal to $\big(0, \bar{\nu}(\tfrac{\overline{x}_1}{\varepsilon})\big).$

Equations \eqref{eq-u-2} and \eqref{bc-u-2},  and equations \eqref{eq-u-3} and \eqref{bc-u-3}
are linear Neumann problems in  $\Upsilon(x_1)$ with respect to the variables $\overline{\xi}_1;$
$x_1$ and $t$ are considered parameters. For the uniqueness, we supply each of these problems with the  condition
\begin{equation}\label{uniq_1}
  \langle u_j(x_1, \cdot ,  t ) \rangle_{\Upsilon(x_1)} :=  \int_{\Upsilon(x_1)} u_j(x_1, \overline{\xi}_1,  t ) \, d\overline{\xi}_1 = 0 \quad \big(j \in \{2, 3\}\big).
\end{equation}

Writing down  the solvability condition for each of these problems, we obtain the following differential equations:
\begin{equation}\label{eq-p-0}
  \partial_{x_1}\big(\lambda(s_0)\,  k_1(x_1) \, \partial_{x_1}p_0\big) = \widehat{Q}(x_1, t), \quad (x_1, t) \in \mathcal{I}^T,
\end{equation}
\begin{equation}\label{eq-p-1}
  \partial_{x_1}\big(\lambda(s_0)\,  k_1(x_1) \, \partial_{x_1}p_1 +  s_1 \, \lambda'(s_0) \,  k_1(x_1)\, \partial_{x_1}p_0\big) = 0, \quad (x_1, t) \in \mathcal{I}^T.
\end{equation}
Here
\begin{equation}\label{hat_phi}
  \widehat{Q}(x_1, t) := \frac{1}{\upharpoonleft\!\!\varpi\!\!\upharpoonright_2} \int_{\partial \Upsilon(x_1)} Q(x_1,\bar{\xi}_1, t)\, d\sigma_{\bar{\xi}_1}.
\end{equation}

Carrying out the same calculations for equation \eqref{Par}  and boundary condition \eqref{bc-cond3}, we obtain the following linear Neumann problems for determining the coefficients $v_2$ and $v_3$ in ansatz \eqref{Anz-S}:
\begin{multline}\label{eq-v-2}
  -\nabla_{\overline{\xi}_1} \boldsymbol{\cdot}\big(\Lambda(s_0) \, \mathbf{K}(x_1,\overline{\xi}_1)\nabla_{\overline{\xi}_1} v_2(x_1,\overline{\xi}_1,t) + \lambda_w(s_0) \, \mathbf{K}(x_1,\overline{\xi}_1)\nabla_{\overline{\xi}_1} u_2\big)
  \\
  =
  \partial_{x_1}\big(\Lambda(s_0) \, k_1 \, \partial_{x_1}s_0 + \lambda_w(s_0) \, k_1 \, \partial_{x_1}p_0\big) - \phi(x_1) \, \partial_t s_0, \quad \overline{\xi}_1 \in \Upsilon(x_1),
\end{multline}
\begin{equation}\label{bc-v-2}
  - \big( \Lambda(s_0) \, \mathbf{K}\nabla_{\overline{\xi}_1} v_2 + \lambda_w(s_0) \, \mathbf{K}\nabla_{\overline{\xi}_1} u_2\big) \boldsymbol{\cdot} \bar{\nu}(\overline{\xi}_1)
      = b(s_0)\, Q(x_1,\overline{\xi}_1,t), \quad \overline{\xi}_1 \in \partial\Upsilon_\varepsilon(x_1),
  \end{equation}
  \begin{equation}\label{v-2-zero}
  \langle v_2(x_1, \cdot ,  t ) \rangle_{\Upsilon(x_1)} =0,
  \end{equation}
and
\begin{multline}\label{eq-v-3}
  -\nabla_{\overline{\xi}_1} \boldsymbol{\cdot}\Big(\Lambda(s_0) \, \mathbf{K}\, \nabla_{\overline{\xi}_1} v_3 +
  s_1 \, \Lambda'(s_0) \, \mathbf{K}\,\nabla_{\overline{\xi}_1} v_2
    +\lambda_w(s_0) \, \mathbf{K}\,\nabla_{\overline{\xi}_1} u_3 +
  s_1 \, \lambda'_w(s_0) \, \mathbf{K}\, \nabla_{\overline{\xi}_1} u_2
   \Big)
  \\
  =
  \partial_{x_1}\Big(\Lambda(s_0) \, k_1 \, \partial_{x_1}s_1 + s_1 \, \Lambda'(s_0) \, k_1 \, \partial_{x_1}s_0
  +
  \lambda_w(s_0) \, k_1 \, \partial_{x_1}p_1 + s_1 \, \lambda'_w(s_0) \, k_1 \, \partial_{x_1}p_0
    \Big)
    \\
    - \phi(x_1) \, \partial_t s_1, \quad \overline{\xi}_1 \in \Upsilon(x_1),
\end{multline}
\begin{multline}\label{bc-v-3}
-  \Big( \Lambda(s_0) \, \mathbf{K} \,\nabla_{\overline{\xi}_1} v_3 + s_1 \, \Lambda'(s_0) \, \mathbf{K}\, \nabla_{\overline{\xi}_1} v_2
   +\lambda_w(s_0) \, \mathbf{K}\, \nabla_{\overline{\xi}_1} u_3
   \\
    +
  s_1 \, \lambda'_w(s_0) \, \mathbf{K}\, \nabla_{\overline{\xi}_1} u_2 \Big) \boldsymbol{\cdot} \bar{\nu}(\overline{\xi}_1)
     = b'(s_0)\, s_1\, Q,  \quad \overline{\xi}_1 \in \partial\Upsilon_\varepsilon(x_1),
  \end{multline}
  \begin{equation}\label{v-3-zero}
  \langle v_3(x_1, \cdot ,  t ) \rangle_{\Upsilon(x_1)} =0.
  \end{equation}

The solvability condition for the  Neumann problem \eqref{eq-v-2}-\eqref{v-2-zero} is given by the differential equation
\begin{equation}\label{eq-s-0}
  \phi \, \partial_t s_0 = \partial_{x_1}\big(\Lambda(s_0) \, k_1 \, \partial_{x_1}s_0 + \lambda_w(s_0) \, k_1 \, \partial_{x_1}p_0\big) -  b(s_0)\, \widehat{Q}(x_1, t), \quad (x_1, t) \in \mathcal{I}^T,
\end{equation}
while for the Neumann problem \eqref{eq-v-3}-\eqref{v-3-zero} it is given by the equation
\begin{multline}\label{eq-s-1}
   \phi \, \partial_t s_1 = \partial_{x_1}\Big(\Lambda(s_0) \, k_1 \, \partial_{x_1}s_1 + s_1 \, \Lambda'(s_0) \,  k_1 \, \partial_{x_1}s_0
   \\
  +
  \lambda_w(s_0) \,  k_1 \, \partial_{x_1}p_1 + s_1 \, \lambda'_w(s_0) \,  k_1 \, \partial_{x_1}p_0
    \Big)
    -  b'(s_0) \, s_1 \, \widehat{Q}(x_1, t), \quad (x_1, t) \in \mathcal{I}^T.
\end{multline}

\subsection{The limit problem}\label{subsect-3-2}

Differential equations \eqref{eq-p-0} and  \eqref{eq-s-0} supplemented with appropriate boundary conditions and the initial one from problem
$(\mathbb{P}_\varepsilon\mathbb{S}_\varepsilon\!)$  form the following system
\begin{equation}\label{limit_prob}
 \left\{\begin{array}{l}
\partial_{x_1}\big(\lambda(s_0)\,  k_1(x_1) \, \partial_{x_1}p_0\big) =  \widehat{Q}(x_1,t)\quad \text{in} \ \ \mathcal{I}^T,
\\[3pt]
 \phi \, \partial_t s_0 = \partial_{x_1}\big(\Lambda(s_0) \, k_1 \, \partial_{x_1}s_0 + \lambda_w(s_0) \, k_1 \, \partial_{x_1}p_0\big) -  b(s_0)\, \widehat{Q}\quad \text{in} \ \ \mathcal{I}^T,
 \\[4pt]
 p_0(0,t) = q_0(t) \quad \text{and} \quad  p_0(\ell,t) = q_\ell(t), \quad t\in [0,T],
 \\[4pt]
 s_0(0,t)  = S^0(0,t) \quad \text{and} \quad  s_0(\ell,t) = S^0(\ell,t), \quad t\in [0,T],
 \\[3pt]
 s_0(x_1,0) = S^0(x_1,0), \quad x_1 \in [0, \ell],
 \end{array}\right.
\end{equation}
with unknown functions $p_0(x_1,t)$ and $s_0(x_1,t),$ $(x_1,t) \in [0,\ell]\times [0,T].$ This problem is called \textit{the limit problem} for problem $(\mathbb{P}_\varepsilon\mathbb{S}_\varepsilon\!)$ for the subcase $\alpha =1 \ \text{and} \ \beta =0.$

From the first differential equation in problem \eqref{limit_prob} and the boundary conditions for $p_0$ we derive the representation
\begin{align}
  p_0(x_1,t)  = & \, \,  \int_{0}^{x_1} \frac{1}{\lambda(s_0)\,  k_1(\eta)} \int_{0}^{\eta} \widehat{Q}(\varsigma,t) \, d\varsigma \, d\eta \notag
  \\[2pt]
  & + \frac{q_\ell(t) -q_0(t) - \int_{0}^{\ell} \frac{1}{\lambda(s_0)\,  k_1} \int_{0}^{\eta} \widehat{Q} \, d\varsigma \, d\eta}{\int_{0}^{\ell} \frac{1}{\lambda(s_0)\,  k_1} \, d\eta} \, \int_{0}^{x_1} \frac{1}{\lambda(s_0)\,  k_1(\eta)} \, d\eta  + q_0(t). \label{representation}
\end{align}
Using \eqref{representation}, the second differential equation in \eqref{limit_prob} is reduced to
\begin{equation}\label{limit_parab-eq}
\phi \, \partial_t s_0 = \partial_{x_1}\big(\Lambda(s_0) \,  k_1(x_1) \, \partial_{x_1}s_0\big) + a(s_0,x_1,t) \, \partial_{x_1} s_0,
\end{equation}
where
$$
a(s_0,x_1,t) = b'(s_0)\, \left(\int_{0}^{x_1} \widehat{Q}(\varsigma,t) \, d\varsigma  + \frac{q_\ell(t) -q_0(t) - \int_{0}^{\ell}\left( \frac{1}{\lambda(s_0)\,  k_1} \int_{0}^{\eta} \widehat{Q} \, d\varsigma \right) d\eta}{\int_{0}^{\ell} \frac{1}{\lambda(s_0)\,  k_1} \, d\eta}\right).
$$

Considering the properties of the function  $\phi,$ the coefficient near  the derivative $\partial_t s_0$ does not introduce any features into  equation \eqref{limit_parab-eq}. Consequently, we can postulate that it is equal to one. Otherwise, we must implement the following replacement: $\vartheta =  \phi \, s_0.$
Therefore, further, we justify  the existence and uniqueness of the solution to the problem
\begin{equation}\label{limit_parab-prob}
 \left\{\begin{array}{l}
 \partial_t s_0 = \partial_{x_1}\big(\Lambda(s_0) \,  \partial_{x_1}s_0\big) + a(s_0,x_1,t) \, \partial_{x_1} s_0 \quad \text{in} \ \ \mathcal{I}^T,
 \\[4pt]
 s_0(0,t)  = S^0(0,t) \quad \text{and} \quad  s_0(\ell,t) = S^0(\ell,t), \quad t\in [0,T],
 \\[3pt]
 s_0(x_1,0) = S^0(x_1,0), \quad x_1 \in [0, \ell].
 \end{array}\right.
\end{equation}
For this we verify the conditions of Theorem 5.2 from \cite[Chapt. VI]{Lad_Sol_Ura_1968}. Due to \eqref{limit-saturation} and \eqref{Lambda1},  the function $\Lambda$ is bounded  below and above by positive constants (this means that the inequalities (5.10) from \cite{Lad_Sol_Ura_1968} with $m=2$ are satisfied).  The inequality (5.9) holds with $\Phi=1.$
Next, according Remark 5.1 from \cite{Lad_Sol_Ura_1968}, we show that the inequality (5.6) holds.
Due to the properties of the functions $p_c,$ $\lambda_0,$ $\lambda_w,$ and $\lambda$ (see \eqref{Leveret} - \eqref{total mobillity}), the derivative $\Lambda'$ is bounded, and recalling properties of the functions $k_1,$ $q_0,$ $q_\ell$ and $Q$, we get the boundedness  of the coefficient $a.$
Thus,
$$
|\Lambda'(s)| \, (1+ |p|)^2 + |\Lambda'(s) \, p^2 +  p\,a(s,x_1,t)| \le  C_1 \, (1+ |p|)^2.
$$
This inequality coincides with  (5.6) for the coefficients of problem  \eqref{limit_parab-prob}.

In addition, the smoothness of the functions $k_1,$ $q_0,$ $q_\ell$ and $Q$ ensures the necessary smoothness of the coefficients $\Lambda$ and $a$ to satisfy the smoothness conditions in Theorem 5.2 from \cite[Chapt. VI]{Lad_Sol_Ura_1968}. The assumption for the smoothness of the function $S^0$ ensure the fulfillment of the compatibility condition of zero order for  problem \eqref{limit_parab-prob}.
To satisfy the compatibility condition of first order, we additionally assume that
\begin{equation}\label{com-cond-1}
  \Big(\partial_t S^0(x_1,0) - \partial_{x_1}\big(\Lambda(S^0(x_1,0)) \,  \partial_{x_1}S^0(x_1,0)\big) \Big)\Big|_{x_1=0, \, \text{and} \, x_1=\ell} =0.
\end{equation}

Thus, based on Theorem 5.2 \cite{Lad_Sol_Ura_1968}, the following statement is true.

\begin{proposition}\label{Prop3-1}
Problem \eqref{limit_parab-prob} has a unique solution in  the H\"older space  $\mathcal{C}^{2+\gamma, 1 +\gamma}([0,\ell]\times[0,T])$ with $\gamma \in (0,1).$
\end{proposition}
This statement and representation \eqref{representation}  mean that there exist a unique classical solution to the limit problem \eqref{limit_prob}.

%%%%%%%%%%%%
\subsection{Determination of the other coefficients in \eqref{Anz-P} and \eqref{Anz-S}}\label{subsect-3-3}

Differential equations in problem \eqref{limit_prob} are the solvability conditions for the Neumann problems  \eqref{eq-u-2}, \eqref{bc-u-2}, \eqref{uniq_1} and \eqref{eq-v-2}-\eqref{v-2-zero} respectively.
Therefore, we first uniquely determine the solution $u_2,$ and then the solution $v_2.$

Taking the first differential equation in \eqref{limit_prob} into account, the Neumann problem \eqref{eq-u-2}, \eqref{bc-u-2}, \eqref{uniq_1} can be now rewritten  as follows:
\begin{equation}\label{Neumann u2}
 \left\{\begin{array}{l}
 - \nabla_{\overline{\xi}_1} \boldsymbol{\cdot}\big( \mathbf{K} \nabla_{\overline{\xi}_1} u_2(x_1,\overline{\xi}_1,t)\big)=
 \dfrac{\widehat{Q}(x_1,t)}{\lambda(s_0(x_1,t))}, \quad \overline{\xi}_1 \in \Upsilon(x_1),
\\[4pt]
-  \big(\mathbf{K}\nabla_{\overline{\xi}_1} u_2\big) \boldsymbol{\cdot} \bar{\nu}(\overline{\xi}_1)  = \dfrac{Q(x_1,\overline{\xi}_1,t)}{\lambda(s_0)}, \quad \overline{\xi}_1 \in \partial\Upsilon_\varepsilon(x_1),
\\[4pt]
\langle u_2(x_1, \cdot ,  t ) \rangle_{\Upsilon(x_1)} =0.
\end{array}\right.
\end{equation}
Since the function $Q$ vanishes if $x_1 \in [0, \delta] \cup  [\ell - \delta, \ell]$ (see the assumptions for $Q$),  we conclude that $u_2(x_1,\overline{\xi}_1,t) = 0$ for those values of $x_1.$ In addition, since $Q|_{t=0}= q_0|_{t=0}= q_\ell|_{t=0}=0,$ the coefficient $u_2$  vanishes also at $t=0.$

Using the first two relations in \eqref{Neumann u2} and the differential equation for $s_0$ in \eqref{limit_prob}, the Neumann problem \eqref{eq-v-2}-\eqref{v-2-zero}  is presented herewith in the following form:
\begin{equation}\label{Neumann v2}
 \left\{\begin{array}{l}
 \nabla_{\overline{\xi}_1} \boldsymbol{\cdot}\big(\Lambda(s_0) \, \mathbf{K}(x_1,\overline{\xi}_1)\nabla_{\overline{\xi}_1} v_2(x_1,\overline{\xi}_1,t)\big) = 0, \quad \overline{\xi}_1 \in \Upsilon(x_1),
\\
\nabla_{\overline{\xi}_1} \boldsymbol{\cdot}\big(\Lambda(s_0) \, \mathbf{K}(x_1,\overline{\xi}_1)\nabla_{\overline{\xi}_1} v_2(x_1,\overline{\xi}_1,t) = 0, \quad \overline{\xi}_1 \in \partial\Upsilon_\varepsilon(x_1),
\\
\langle v_2(x_1, \cdot ,  t ) \rangle_{\Upsilon(x_1)} =0.
\end{array}\right.
\end{equation}
Obviously, $v_2 \equiv 0.$

For the coefficients $p_1$ and $s_1$ we obtain the following linear system of differential equations:
\begin{equation}\label{limit_parab-prob+1}
 \left\{\begin{array}{l}
 \partial_{x_1}\big(\lambda(s_0)\,  k_1 \, \partial_{x_1}p_1 +  s_1 \, \lambda'(s_0) \,  k_1\, \partial_{x_1}p_0\big) = 0 \quad \text{in} \ \mathcal{I}^T,
 \\[4pt]
  \phi \, \partial_t s_1 = \partial_{x_1}\Big(\Lambda(s_0) \,  k_1 \, \partial_{x_1}s_1 +  a_1\, s_1  +  a_2 \, \partial_{x_1}p_1\Big)
  -  b'(s_0) \,  \widehat{Q}\, s_1 \quad \text{in} \ \mathcal{I}^T,
  \\[4pt]
 p_1(0,t) = 0 \quad \text{and} \quad  p_1(\ell,t) = 0, \quad t\in [0,T],
 \\[4pt]
  s_1(0,t)  = 0 \quad \text{and} \quad  s_1(\ell,t) = 0, \quad t\in [0,T],
 \\[3pt]
 s_1(x_1,0) = 0, \quad x_1 \in [0, \ell],
 \end{array}\right.
\end{equation}
where
\begin{equation}\label{coeff-a1-a2}
  a_1(x_1,t)  = \Lambda'(s_0) \, k_1\, \partial_{x_1}s_0 + \lambda'_w(s_0) \, k_1\, \partial_{x_1}p_0,\quad
  a_2(x_1,t)  = \lambda_w(s_0) \, k_1(x_1).
 \end{equation}

 It is easy to see that this problem has the trivial solution $p_1 \equiv s_1 \equiv 0.$ It remains to show that the solution is unique. From the first differential equation in \eqref{limit_parab-prob+1} and the corresponding boundary conditions it follows that
 \begin{equation}\label{pre-p1}
 p_1(x_1,t) = d_1(t, s_1) \int_{0}^{x_1} \frac{d\eta}{\lambda(s_0(\eta,t))\,  k_1(\eta)} - \int_{0}^{x_1} \frac{s_1(\eta,t) \, \lambda'(s_0(\eta,t)) \, \partial_{x_1}p_0(\eta,t)}{\lambda(s_0(\eta,t))} \, d\eta,
 \end{equation}
 where
 $$
 d_1(t, s_1) = \int_{0}^{\ell} \frac{s_1 \, \lambda'(s_0) \, \partial_{x_1}p_0}{\lambda(s_0)} \, d\eta \, \bigg(\int_{0}^{\ell} \frac{d\varsigma}{\lambda(s_0)\, k_1}\bigg)^{-1}.
 $$
 Substituting \eqref{pre-p1}  into the second differential equation of problem \eqref{limit_parab-prob+1}, we get
\begin{equation}\label{eq-s1}
    \phi \, \partial_t s_1 = \partial_{x_1}\Big(\Lambda(s_0) \,  k_1 \, \partial_{x_1}s_1  + A_1(x_1,t) \, s_1\Big) + \int_{0}^{\ell} A_2(x_1,\eta, t) \, s_1(\eta,t)\, d\eta + A_3(x_1,t)\, s_1,
\end{equation}
where
\begin{gather*}
  A_1(x_1,t) = a_1(x_1,t) -  \frac{\lambda'(s_0) \,   k_1 \, \partial_{x_1}p_0}{\lambda(s_0)}, \qquad A_3(x_1,t) = -    b'(s_0) \,  \widehat{Q}
      \\
  A_2(x_1,\eta,t) = \bigg(\int_{0}^{\ell} \frac{d\varsigma}{\lambda(s_0)\,  k_1}\bigg)^{-1} b'(s_0(x_1,t)) \, \partial_{x_1} s_0(x_1,t) \, \frac{ \lambda'(s_0(\eta,t)) \, \partial_{x_1}p_0(\eta,t)}{\lambda(s_0(\eta,t))}.
\end{gather*}
In virtue of the assumptions for the coefficients of problem $(\mathbb{P}_\varepsilon\mathbb{S}_\varepsilon\!)$ and the smoothness  of the solution $(p_0, s_0)$ to the limit problem \eqref{limit_prob}, the functions $A_1, A_2, A_3$ are bounded in the uniform pointwise norm. Taking this into account, multiplying equation \eqref{eq-s1} by $s_1,$  and integrating by parts over the domain $(0, \tau)\times (0, \ell),$ where $\tau$ is arbitrary number from $(0,T),$ we get the inequality
\begin{multline}\label{in-1}
c_0 \int_{0}^{\ell} s_1^2(x_1,\tau)\, dx_1 + c_1 \int_{0}^{\tau}\int_{0}^{\ell} |\partial_{x_1} s_1|^2\, dx_1\, dt
\\
\le
c_2 \int_{0}^{\tau}\int_{0}^{\ell} |\partial_{x_1} s_1| \, |s_1|\, dx_1\, dt + c_3 \int_{0}^{\tau} \bigg(\int_{0}^{\ell}|s_1|\, dx_1\bigg)^2 dt + c_4
\int_{0}^{\tau}\int_{0}^{\ell} |s_1|^2\, dx_1\, dt.
\end{multline}
Applying  Cauchy's inequality with $\delta = \frac{c_1}{2} \ (ab \leq \delta a^2 + \tfrac{b^2}{4\delta})$ to  the first integral in the right-hand side of \eqref{in-1} and  the Cauchy-Bunyakovsky-Schwarz inequality to the second integral,   we deduce
\begin{equation}\label{in-2}
  c_0 \int_{0}^{\ell} s_1^2(x_1,\tau)\, dx_1 + \frac{c_1}{2} \int_{0}^{\tau}\int_{0}^{\ell} |\partial_{x_1} s_1|^2\, dx_1\, dt
\le c_5 \int_{0}^{\tau}\int_{0}^{\ell} |s_1|^2\, dx_1\, dt.
\end{equation}
Then Gronwall's lemma, applied to the inequality
$$
\int_{0}^{\ell} s_1^2(x_1,\tau)\, dx_1 \le \frac{c_5}{c_0} \, \int_{0}^{\tau}\int_{0}^{\ell} |s_1|^2\, dx_1\, dt,
$$
yields $\int_{0}^{\ell} s_1^2(x_1,\tau)\, dx_1 =0$ for any $\tau \in (0,T).$ Thus, $s_1 \equiv 0,$ and from \eqref{pre-p1} it follows that $p_1 \equiv  0.$

\begin{remark}\label{remark-3-2}
 If boundary conditions or right-hand sides in the differential equations of problem \eqref{limit_parab-prob+1} are non-zero, then
 Gronwall's lemma gives an a priori estimate for the solution, and Galerkin's method proves its existence.
 For  more general elliptic-parabolic problems, the existence and uniqueness of solutions were proved in
 \cite{Alt-Luck-1983,Bokalo-2017}.
\end{remark}

Next, it follows from the Neumann problem \eqref{eq-u-3}, \eqref{bc-u-3}, \eqref{uniq_1} that the solution $u_3 $ is equal to zero, and from problem \eqref{eq-v-3}-\eqref{v-3-zero} that  $v_3 \equiv 0.$

As a result, the approximation functions \eqref{Anz-P} and \eqref{Anz-S} for the solution to problem $(\mathbb{P}_\varepsilon\mathbb{S}_\varepsilon\!)$ now look as follows
\begin{equation}\label{AppFunctions}
  \mathfrak{P}_\varepsilon(x,t) = p_0(x_1,t) +  \varepsilon^2 u_2\Big(x_1, \dfrac{\overline{x}_1}{\varepsilon}, t \Big)
\quad \text{and} \quad 
\mathfrak{S}_\varepsilon(x,t) = s_0(x_1,t).
\end{equation}
Moreover, from the above calculations it follows that
\begin{gather}\label{eq-P-e}
\nabla \boldsymbol{\cdot} \big(\lambda(s_0) \,  \mathbb{K}_\varepsilon(x)  \nabla \mathfrak{P}_\varepsilon \big) = \varepsilon^2 \mathcal{F}^{(1)}_\varepsilon
  \quad \text{in} \ \ \Omega_\varepsilon^T,
  \\ \label{eq-S-e}
 \phi(x_1) \, \partial_t  s_0 -  \nabla \boldsymbol{\cdot} \big(\Lambda(s_0) \,  \mathbb{K}_\varepsilon(x) \nabla s_0  -  b(s_0) \, \overrightarrow{\mathfrak{V}}_\varepsilon \big) = \varepsilon^2\mathcal{F}^{(2)}_\varepsilon  \quad \text{in} \ \ \Omega_\varepsilon^T,
\end{gather}
\begin{gather}
\label{bc-P-e}
\Big( \lambda(s_0) \,  \mathbb{K}_\varepsilon(x)  \nabla \mathfrak{P}_\varepsilon\Big)  \boldsymbol{\cdot} {\boldsymbol{\nu}}_\varepsilon + \varepsilon\,  Q_\varepsilon = 0 \quad \text{on} \ \  \Gamma^T_\varepsilon,
\\ \label{bc-S-e}
\Big(\Lambda(s_0) \,  \mathbb{K}_\varepsilon(x) \nabla s_0  -  b(s_0) \, \overrightarrow{\mathfrak{V}}_\varepsilon\Big) \boldsymbol{\cdot} {\boldsymbol{\nu}}_\varepsilon + \varepsilon\, b(s_0) \,
Q_\varepsilon(x,t)= 0 \quad \text{on} \ \  \Gamma^T_\varepsilon,
\end{gather}
\begin{gather}
  \label{bc-P-e}
  \mathfrak{P}_\varepsilon|_{(x,t)\in \Upsilon^T_\varepsilon(0)} = q_0(t) \quad \text{and} \quad  \mathfrak{P}_\varepsilon|_{(x,t)\in \Upsilon^T_\varepsilon(\ell)} = q_\ell(t),
\\
\label{bc-S-e}
  s_0|_{(x,t)\in \Upsilon^T_\varepsilon(0)} = S^0(0,t), \quad   s_0|_{(x,t)\in \Upsilon^T_\varepsilon(\ell)} = S^0(\ell,t), \quad s_0|_{t=0} = S^0(x_1,0),
  \end{gather}
where the vector-function $ \overrightarrow{\mathfrak{V}}_\varepsilon = -  \lambda(s_0) \,  \mathbb{K}_\varepsilon(x)  \nabla \mathfrak{P}_\varepsilon,$ 
$$
\mathcal{F}^{(1)}_\varepsilon =  \partial_{x_1}(\lambda(s_0)\, k_1\, \partial_{x_1} u_2)\quad \text{and} \quad
\mathcal{F}^{(2)}_\varepsilon = - \partial_{x_1}(\lambda_w(s_0)\, k_1 \, \partial_{x_1} u_2).
$$
Due to the smoothness of the functions $k_1,$ $\lambda,$ $\lambda_w,$ $s_0$ (see Proposition~\ref{Prop3-1}) and $u_2,$ we have
\begin{equation}\label{est-F}
  \sup_{\Omega_\varepsilon^T}|\mathcal{F}^{(1)}_\varepsilon(x,t)| \le C_1, \quad
  \sup_{\Omega_\varepsilon^T}|\mathcal{F}^{(2)}_\varepsilon(x,t)| \le C_2.
\end{equation}

\subsection{Justification}
Here, we present and prove one of the main results.

\begin{theorem}[$\alpha = 1, \, \beta =0$]\label{Th_1}
Assume that, in addition to the main assumptions made in Section~\ref{Sec:Statement},
condition \eqref{com-cond-1} is  satisfied.
Then,  there are positive constants $\tilde{C}_0,$ $\tilde{C}_1,$ $\tilde{C}_2$ and $\varepsilon_0$ such that for all $\varepsilon\in (0, \varepsilon_0)$
 \begin{equation}\label{app-estimate1}
 \tfrac{1}{\sqrt{\upharpoonleft  \Omega_\varepsilon \! \upharpoonright_3}}\,  \max_{t\in \times [0,T]}
 {\|P_\varepsilon(\cdot,t) - \mathfrak{P}_\varepsilon(\cdot,t) \|}_{H^1(\Omega_\varepsilon)}  \le \tilde{C}_0 \, \varepsilon^{2},
\end{equation}
\begin{equation}\label{app-estimate2}
\tfrac{1}{\sqrt{\upharpoonleft  \Omega_\varepsilon \! \upharpoonright_3}}\,  \max_{t\in \times [0,T]}
{\| S_\varepsilon(\cdot,t) - s_0(\cdot,t) \|}_{L^2(\Omega_\varepsilon)}
\le  \tilde{C}_1\, \varepsilon^2,
\end{equation}
\begin{equation}\label{app-estimate3}
\tfrac{1}{\sqrt{\upharpoonleft  \Omega_\varepsilon \! \upharpoonright_3}}\,
{\|S_\varepsilon - s_0\|}_{L^2(0,T; H^1(\Omega_\varepsilon))} \le
 \tilde{C}_2\, \varepsilon^2.
\end{equation}
 \end{theorem}
 \begin{remark}
  Since  inequalities \eqref{app-estimate1} and \eqref{app-estimate2} contain  integral norms  and the integrals taken  over the thin cylinder $\Omega_\varepsilon,$ these norms must be rescaled, namely divided by the square of the volume of  $\Omega_\varepsilon.$
 \end{remark}
\begin{proof}
First, we specify relations for the differences
\begin{equation}\label{differences-U-W}
U_\varepsilon := \mathfrak{P}_\varepsilon - P_\varepsilon, \qquad W_\varepsilon := s_0 - {S}_\varepsilon.
\end{equation}

It is easy to see from \eqref{bc-cond1}, \eqref{bc-cond1+}, \eqref{initial-cond} and  \eqref{bc-P-e}, \eqref{bc-S-e} that
\begin{gather}
  \label{bc-U-e}
  U_\varepsilon|_{(x,t)\in \Upsilon^T_\varepsilon(0)} = 0 \quad \text{and} \quad  U_\varepsilon|_{(x,t)\in \Upsilon^T_\varepsilon(\ell)} = 0,
\\
\label{bc-W-e}
  W_\varepsilon|_{(x,t)\in \Upsilon^T_\varepsilon(0)} = 0, \quad   W_\varepsilon|_{(x,t)\in \Upsilon^T_\varepsilon(\ell)} = 0, \quad W_\varepsilon|_{t=0} = 0.
  \end{gather}

Subtracting  \eqref{Ell} from \eqref{eq-P-e} and using the mean value theorem, we get
\begin{equation}\label{eq-U-e}
\nabla \boldsymbol{\cdot} \Big(\lambda({S}_\varepsilon(x,t)) \,  \mathbb{K}_\varepsilon(x)  \nabla U_\varepsilon  + W_\varepsilon \, \lambda'(\Theta^{(1)}_\varepsilon)\, \mathbb{K}_\varepsilon(x)  \nabla\mathfrak{P}_\varepsilon \Big) =
 \varepsilon^2 \mathcal{F}^{(1)}_\varepsilon
  \quad \text{in} \ \ \Omega_\varepsilon^T,
\end{equation}
where $\Theta^{(1)}_\varepsilon(x,t) = \theta_1 \, s_0(x_1,t) + (1-\theta_1)\, {S}_\varepsilon(x,t), \quad \theta_1 \in (0,1).$

Similarly, subtracting equation \eqref{Par} from \eqref{eq-S-e}, we deduce that
\begin{multline}\label{eq-W-e}
 \phi(x_1) \, \partial_t  W_\varepsilon -  \nabla \boldsymbol{\cdot} \Big(\Lambda({S}_\varepsilon) \,  \mathbb{K}_\varepsilon(x) \nabla W_\varepsilon + W_\varepsilon \, \Lambda'(\Theta^{(2)}_\varepsilon)\, \mathbb{K}_\varepsilon(x)  \nabla s_0
\\
 + \lambda_w({S}_\varepsilon) \,  \mathbb{K}_\varepsilon(x) \nabla U_\varepsilon
 + W_\varepsilon \, \lambda'_w(\Theta^{(3)}_\varepsilon)\, \mathbb{K}_\varepsilon(x)  \nabla\mathfrak{P}_\varepsilon
 \Big)
  =  \varepsilon^2\mathcal{F}^{(2)}_\varepsilon  \quad \text{in} \ \ \Omega_\varepsilon^T.
\end{multline}
The Neumann type boundary conditions for $U_\varepsilon$ and $W_\varepsilon$ on $\Gamma^T_\varepsilon$ are as follows
\begin{equation}\label{Nbc-U-e}
\Big(\lambda({S}_\varepsilon) \,  \mathbb{K}_\varepsilon(x)  \nabla U_\varepsilon +
 W_\varepsilon \, \lambda'(\Theta^{(1)}_\varepsilon)\, \mathbb{K}_\varepsilon(x)  \nabla\mathfrak{P}_\varepsilon
\Big) \boldsymbol{\cdot} {\boldsymbol{\nu}}_\varepsilon
  =  0 ,
\end{equation}
and
\begin{equation} \label{Nbc-W-e}
\Big(\Lambda({S}_\varepsilon) \,  \mathbb{K}_\varepsilon(x) \nabla W_\varepsilon
+ W_\varepsilon \, \Lambda'(\Theta^{(2)}_\varepsilon)\, \mathbb{K}_\varepsilon(x)  \nabla s_0
 + \lambda_w({S}_\varepsilon) \,  \mathbb{K}_\varepsilon(x) \nabla U_\varepsilon
 + W_\varepsilon \, \lambda'_w(\Theta^{(3)}_\varepsilon)\, \mathbb{K}_\varepsilon(x)  \nabla\mathfrak{P}_\varepsilon \Big) \boldsymbol{\cdot} {\boldsymbol{\nu}}_\varepsilon
=  0.
\end{equation}

Now, we multiply equation \eqref{eq-U-e} by $U_\varepsilon$ and integrate it over $\Omega_\varepsilon.$  Given \eqref{bc-U-e} and \eqref{Nbc-U-e}, we get
$$
\int_{\Omega_\varepsilon} \lambda({S}_\varepsilon) \,  \mathbb{K}_\varepsilon  \nabla U_\varepsilon \boldsymbol{\cdot} \nabla U_\varepsilon \, dx = - \int_{\Omega_\varepsilon} W_\varepsilon \, \lambda'(\Theta^{(1)}_\varepsilon)\, \mathbb{K}_\varepsilon  \nabla\mathfrak{P}_\varepsilon \boldsymbol{\cdot} \nabla U_\varepsilon \, dx -  \varepsilon^2 \int_{\Omega_\varepsilon} \mathcal{F}^{(1)}_\varepsilon
\, U_\varepsilon\, dx.
$$
Taking into account \eqref{cercitive}, \eqref{total mobillity}, \eqref{est-F} and the smoothness of $p_0$ and $u_2,$ from this equality we derive
\begin{equation}\label{in-U-e}
  \int_{\Omega_\varepsilon} |\nabla U_\varepsilon|^2\, dx \le C_1 \Big(\int_{\Omega_\varepsilon} |W_\varepsilon|\, |\nabla U_\varepsilon|\, dx +
  \varepsilon^2 \int_{\Omega_\varepsilon} |U_\varepsilon|\, dx\Big).
\end{equation}
Due to \eqref{bc-U-e} we have
$$
\int_{\Omega_\varepsilon} |U_\varepsilon|\, dx \le C_2 \, \varepsilon\, \sqrt{\int_{\Omega_\varepsilon}U^2_\varepsilon\, dx} \le C_3 \, \varepsilon\,
\sqrt{\int_{\Omega_\varepsilon}|\nabla U_\varepsilon|^2\, dx}.
$$
As a result, from \eqref{in-U-e} we get
\begin{equation}\label{in-U-e+1}
  \int_{\Omega_\varepsilon} |\nabla U_\varepsilon(x,t)|^2\, dx \le C_4 \Big( \varepsilon^6 + \int_{\Omega_\varepsilon} W^2_\varepsilon(x,t)  dx\Big)
  \quad \text{for any} \ \ t \in [0,T].
\end{equation}

Multiplying  \eqref{eq-W-e} by $W_\varepsilon$ and integrating  it over $\Omega_\varepsilon^\tau,$ where $\tau \in (0,T),$ and
taking  into account \eqref{cercitive}, \eqref{regular} and the assumption for the porosity,  we obtain
\begin{equation}\label{ine-W-e}
  \int\limits_{\Omega_\varepsilon} W^2_\varepsilon(x,\tau)\, dx + \int\limits_{\Omega_\varepsilon^\tau} |\nabla W_\varepsilon|^2 \, dxdt
        \le C_0 \bigg(|I_1(\varepsilon)| + |I_2(\varepsilon)| + \varepsilon^2 \int\limits_{\Omega_\varepsilon^\tau} \mathcal{F}^{(2)}_\varepsilon
\, W_\varepsilon\, dxdt\bigg),
\end{equation}
where
\begin{gather}
  I_1(\varepsilon) = - \int_{\Omega_\varepsilon^\tau}
  W_\varepsilon \Big(\Lambda'(\Theta^{(2)}_\varepsilon)\, \mathbb{K}_\varepsilon(x)  \nabla s_0
  +  \lambda'_w(\Theta^{(3)}_\varepsilon)\, \mathbb{K}_\varepsilon(x)  \nabla\mathfrak{P}_\varepsilon\Big)
   \boldsymbol{\cdot} \nabla W_\varepsilon \label{I-1}
  \, dxdt,
   \\
  I_2(\varepsilon) = - \int_{\Omega_\varepsilon^\tau}
  \lambda_w({S}_\varepsilon) \,  \mathbb{K}_\varepsilon(x) \nabla U_\varepsilon
   \boldsymbol{\cdot} \nabla W_\varepsilon  \, dx\, dt. \label{I-2}
\end{gather}
Using the smoothness of the coefficients of problem $(\mathbb{P}_\varepsilon\mathbb{S}_\varepsilon\!),$  with the help of  Cauchy's inequality with $\gamma >0 \ (ab \leq \gamma a^2 + \tfrac{b^2}{4\gamma}),$  we derive that
\begin{equation}\label{est-I+1}
|I_1(\varepsilon)| \le \gamma_1  \int_{0}^{\tau}\int_{\Omega_\varepsilon} | \nabla W_\varepsilon|^2  \, dx\, dt + \frac{C_1}{\gamma_1}
\int_{0}^{\tau}\int_{\Omega_\varepsilon} W_\varepsilon^2  \, dx\, dt,
\end{equation}
  \begin{align}\label{est-I+2}
|I_2(\varepsilon)|  \le  & \  \gamma_2  \int_{0}^{\tau}\int_{\Omega_\varepsilon} | \nabla W_\varepsilon|^2  \, dx\, dt + \frac{C_2}{\gamma_2}
\int_{0}^{\tau}\int_{\Omega_\varepsilon}|\nabla U_\varepsilon|^2  \, dx\, dt \notag
 \\
     &
     \stackrel{\eqref{in-U-e+1}}{\le}
\gamma_2  \int_{0}^{\tau}\int_{\Omega_\varepsilon} | \nabla W_\varepsilon|^2  \, dx\, dt + \frac{C_3}{\gamma_2}
\Big( \varepsilon^6 + \int_{0}^{\tau}\int_{\Omega_\varepsilon} W^2_\varepsilon \, dx dt\Big). 
  \end{align}

 Thanks to  \eqref{est-F} we have
 \begin{equation}\label{est-F-2}
   \bigg|\varepsilon^2 \int_{0}^{\tau}\int_{\Omega_\varepsilon} \mathcal{F}^{(2)}_\varepsilon
\, W_\varepsilon\, dx\, dt\bigg| \le C_4\Big(\varepsilon^6 + \int_{0}^{\tau}\int_{\Omega_\varepsilon} W^2_\varepsilon \, dx dt\Big).
 \end{equation}

Then, taking suitable $\gamma_1$ and $\gamma_2$ in  \eqref{est-I+1} and \eqref{est-I+2}, we derive from \eqref{ine-W-e} the inequality
\begin{equation}\label{ine-W-e+2}
  \int_{\Omega_\varepsilon} W^2_\varepsilon(x,\tau)\, dx + \int_{0}^{\tau} \int_{\Omega_\varepsilon} |\nabla W_\varepsilon|^2 \, dx\, dt
    \le C_5 \Big(\varepsilon^6 + \int_{0}^{\tau}\int_{\Omega_\varepsilon} W^2_\varepsilon(x,t) \, dx dt\Big)
\end{equation}
for any $\tau\in (0,T).$  Applying Gronwall's lemma to the inequality
$$
\int_{\Omega_\varepsilon} W^2_\varepsilon(x,\tau)\, dx  \le C_6 \Big(\varepsilon^6 + \int_{0}^{\tau}\int_{\Omega_\varepsilon} W^2_\varepsilon(x,t) \, dx dt\Big),
$$
we obtain
\begin{equation}\label{uniform-est-W-e}
  \max_{t\in [0,T]} \int_{\Omega_\varepsilon} W^2_\varepsilon(x,t)\, dx \le C_7 \, \varepsilon^6,
\end{equation}
where constant $C_7$ depends on $T.$

To complete the proof, it should be noted that inequalities \eqref{uniform-est-W-e} and \eqref{ine-W-e+2} imply  estimates \eqref{app-estimate2} and \eqref{app-estimate3}, and
\eqref{uniform-est-W-e} and \eqref{in-U-e+1} imply estimate \eqref{app-estimate1}.
\end{proof}

Denote by
$$
\langle u \rangle_{\varepsilon\varpi}(x_1)  := \frac{1}{\varepsilon^2 \upharpoonleft\!\!  \varpi \! \!\upharpoonright_2 } \int_{\varepsilon  \varpi} u(x_1, \overline{x}_1)\, d\overline{x}_1
$$
the mean of a function $u\in L^2(\Omega_\varepsilon)$ over the cross-section $\varepsilon  \varpi$ of the thin cylinder $\Omega_\varepsilon.$

Thanks to \eqref{uniq_1}, it is easy to verify that
\begin{equation}
  \langle \mathfrak{P}_\varepsilon \rangle_{\varepsilon\varpi} = p_0(x_1,t), \quad  \langle \partial_{x_1}\mathfrak{P}_\varepsilon \rangle_{\varepsilon\varpi} = \partial_{x_1}\langle \mathfrak{P}_\varepsilon \rangle_{\varepsilon\varpi} = \partial_{x_1}p_0(x_1,t). \label{m-P-e}
\end{equation}

The subsequent corollary is derived from Theorem~\ref{Th_1}.

\begin{corollary}[the case $\alpha = 1, \, \beta =0$]\label{Corollary3-1}
 Let $(P_\varepsilon, S_\varepsilon)$ be a weak solution to problem $(\mathbb{P}_\varepsilon\mathbb{S}_\varepsilon\!),$ and
 $(p_0, s_0)$ be a solution to the limit problem \eqref{limit_prob}.  Then the following estimates hold:
 \begin{gather}\label{est-P-p}
   \max_{t\in[0,T]}{\| \langle {P}_\varepsilon \rangle_{\varepsilon\varpi} - p_0\|}_{H^1(0, \ell)} + {\| \langle {P}_\varepsilon \rangle_{\varepsilon\varpi} - p_0\|}_{C\left([0, \ell]\times[0,T]\right)} \le C_1 \, \varepsilon^2,
   \\[2pt]
   {\| \langle {S}_\varepsilon \rangle_{\varepsilon\varpi} - s_0\|}_{L^2(0,T; H^1(0, \ell))} + {\| \langle {S}_\varepsilon \rangle_{\varepsilon\varpi} - s_0\|}_{L^2(0,T; C([0, \ell]))} \le C_2 \, \varepsilon^2. \label{est-S-s}
   \end{gather}
 \end{corollary}
\begin{proof}  Using  \eqref{m-P-e} and the Cauchy--Buniakovskii--Schwarz inequality, we deduce
\begin{align*}
{ \| \langle {P}_\varepsilon \rangle_{\varepsilon\varpi}(\cdot,t) - p_0(\cdot,t)\|^2}_{L^2(0, \ell)} = & \  \frac{1}{\varepsilon^4 \upharpoonleft\!\!  \varpi \! \!\upharpoonright^2_2 }\, \int_{0}^{\ell} \bigg(\int_{\varepsilon  \varpi} \big(P_\varepsilon - \mathfrak{P}_\varepsilon\big)\,  d\overline{x}_1\bigg)^2dx_1
 \\
  \le & \  \frac{1}{\varepsilon^2 \upharpoonleft\!\!  \varpi \! \!\upharpoonright_2 }\, \int_{0}^{\ell} \int_{\varepsilon  \varpi} \big(P_\varepsilon - \mathfrak{P}_\varepsilon\big)^2\,  d\overline{x}_1dx_1
  \\
&  \stackrel{\eqref{app-estimate1}}{\le}  \frac{1}{\varepsilon^2 \upharpoonleft \!\!  \varpi \! \!\upharpoonright_2 }\,
  \upharpoonleft \! \Omega_\varepsilon\! \! \upharpoonright_3 \, \tilde{C}^2_0 \, \varepsilon^{4} = C^2_1 \, \varepsilon^4
  \end{align*}
for any $ t\in [0,T].$ Similarly, we demonstrate that
$$
\max_{t\in[0,T]}{ \| \partial_{x_1}\langle {P}_\varepsilon \rangle_{\varepsilon\varpi} - \partial_{x_1}p_0\|}_{L^2(0, \ell)} \le C_1 \, \varepsilon^2.
$$
Considering  the continuous embedding of the space $H^1(0,\ell)$ into $C([0,\ell]),$ we obtain \eqref{est-P-p}.
In the same way, but now using \eqref{app-estimate2}, we derive \eqref{est-S-s}.
\end{proof}

%%%%%%%%%%%%%%%%%%%%
\section{The case  $\alpha =1$ and $\beta <2,$ $\beta \neq  0$}\label{Sect:4}

Now the principal challenge is to understand how the appearance of the value $\varepsilon^\beta$ into the absolute permeability
tensor~\eqref{permeability} affects the overall result. To achieve this, we propose the use of the following approximating functions:
\begin{equation}\label{Anz-P-be}
 \mathfrak{P}_\varepsilon(x,t) = p_0(x_1,t)  + \varepsilon^{2 - \beta} u_{2- \beta}\Big(x_1, \dfrac{\overline{x}_1}{\varepsilon}, t \Big) \quad \text{and}
\quad
 \mathfrak{S}_\varepsilon(x,t) = s_0(x_1,t),
\end{equation}
where $(p_0, s_0)$ is  a solution to problem \eqref{limit_prob} and  $u_{2 - \beta}$  is a solution to problem \eqref{Neumann u2}.

\begin{theorem}[the case $\alpha =1,$ $\beta < 2$ and $\beta \neq  0$]\label{Th_1+}
Assume that, in addition to the main assumptions made in Section~\ref{Sec:Statement},
condition \eqref{com-cond-1} is  satisfied.
Then,  there are positive constants $\tilde{C}_0,$ $\tilde{C}_1,$ $\tilde{C}_2$ and $\varepsilon_0$ such that for all $\varepsilon\in (0, \varepsilon_0)$
the  differences $U_\varepsilon := \mathfrak{P}_\varepsilon - P_\varepsilon$ and $W_\varepsilon := \mathfrak{S}_\varepsilon - {S}_\varepsilon$
 satisfy the estimates
   \begin{equation}\label{app-estimate-U-e-be}
 \tfrac{1}{\sqrt{\upharpoonleft  \Omega_\varepsilon \! \upharpoonright_3}}\,  \max_{t\in \times [0,T]}
 \sqrt{\int_{\Omega_\varepsilon} |\partial_{x_1} U_\varepsilon(\cdot,t)|^2\, dx + \varepsilon^\beta \, \int_{\Omega_\varepsilon}
|\nabla_{\overline{x}_1} U_\varepsilon(\cdot,t)|^2\, dx} \le \tilde{C}_0 \, \varepsilon^{1 - \frac{\beta}{2}},
 \end{equation}
 \begin{equation}\label{app-estimate-W-ep-be}
\tfrac{1}{\sqrt{\upharpoonleft  \Omega_\varepsilon \! \upharpoonright_3}}\,  \max_{t\in \times [0,T]}
{\| W_\varepsilon(\cdot,t) \|}_{L^2(\Omega_\varepsilon)} \le \tilde{C}_1 \,  \varepsilon^{1 - \frac{\beta}{2}},
\end{equation}
and
\begin{equation}\label{app-estimateW-ep2-be}
\tfrac{1}{\sqrt{\upharpoonleft  \Omega_\varepsilon \! \upharpoonright_3}}\,
\sqrt{ \int_{0}^{T} \int_{\Omega_\varepsilon} |\partial_{x_1} W_\varepsilon|^2 \, dx dt
+ \varepsilon^\beta \int_{0}^{T} \int_{\Omega_\varepsilon} |\nabla_{\overline{x}_1} W_\varepsilon|^2 \, dx dt} \le
 \tilde{C}_2\, \varepsilon^{1 - \frac{\beta}{2}}.
\end{equation}
 \end{theorem}
\begin{proof}
Compared to the proof of the previous theorem, the proof of this theorem requires certain changes, as shown below.

Upon substitution of  functions \eqref{Anz-P-be}  into problem $(\mathbb{P}_\varepsilon\mathbb{S}_\varepsilon\!),$ identical relations are found as those presented in \eqref{eq-P-e} -- \eqref{bc-S-e}. However, the orders of the residuals in the  differential equations \eqref{eq-P-e} and \eqref{eq-S-e} will differ; they will now be of order $\mathcal{O}(\varepsilon^{2-\beta}).$
Also $U_\varepsilon $ and $W_\varepsilon $
will satisfy relations \eqref{bc-U-e} -- \eqref{Nbc-W-e} with the corresponding residuals in \eqref{eq-U-e} and \eqref{eq-W-e}.
Multiplying the corresponding equation \eqref{eq-U-e} by $U_\varepsilon$ and integrating it over $\Omega_\varepsilon,$ we get
\begin{equation}\label{U-ine1-be}
\int_{\Omega_\varepsilon} \lambda({S}_\varepsilon) \,  \mathbb{K}_\varepsilon  \nabla U_\varepsilon \boldsymbol{\cdot} \nabla U_\varepsilon \, dx = - \int_{\Omega_\varepsilon} W_\varepsilon \, \lambda'(\Theta^{(1)}_\varepsilon)\, \mathbb{K}_\varepsilon  \nabla\mathfrak{P}_\varepsilon \boldsymbol{\cdot} \nabla U_\varepsilon \, dx
     -  \varepsilon^{2 - \beta} \int_{\Omega_\varepsilon} \mathcal{F}^{(1)}_{\varepsilon}
\, U_\varepsilon\, dx.
\end{equation}
Taking into account the form of the  absolute permeability tensor \eqref{permeability} for this case,
\eqref{cercitive} and \eqref{total mobillity}, the left-hand side in \eqref{U-ine1-be} is estimated from below by the sum of two integrals
$
c \Big(\int_{\Omega_\varepsilon} |\partial_{x_1} U_\varepsilon|^2\, dx + \varepsilon^\beta \, \int_{\Omega_\varepsilon}
|\nabla_{\overline{x}_1} U_\varepsilon|^2\, dx\Big).
$
Thus, we deduce from \eqref{U-ine1-be} the inequality
\begin{equation}\label{Uin2-be}
  c \bigg(\int_{\Omega_\varepsilon} |\partial_{x_1} U_\varepsilon|^2\, dx + \varepsilon^\beta \, \int_{\Omega_\varepsilon}
|\nabla_{\overline{x}_1} U_\varepsilon|^2\, dx\bigg)
 \le \bigg|\int_{\Omega_\varepsilon} W_\varepsilon \, \lambda'(\Theta^{(1)}_\varepsilon)\, \mathbb{K}_\varepsilon  \nabla\mathfrak{P}_\varepsilon \boldsymbol{\cdot} \nabla U_\varepsilon \, dx\bigg|
  +  \varepsilon^{2 - \beta} \int_{\Omega_\varepsilon} |\mathcal{F}^{(1)}_{\varepsilon}|\, |U_\varepsilon|\, dx.
\end{equation}
The estimation of the right side of \eqref{Uin2-be} will now be undertaken in the following manner. The first integral is less than
\begin{multline*}
  C_2 \int_{\Omega_\varepsilon} |W_\varepsilon| \, \Big|\big(k_1\, \partial_{x_1}p_0 + \varepsilon^{2-\beta} \partial_{x_1}u_{2-\beta}\big) \partial_{x_1}U_\varepsilon + \varepsilon (\mathbf{K} \nabla_{\overline{\xi}_1}u_{2-\beta}\big)\cdot \nabla_{\overline{x}_1} U_\varepsilon\Big)\bigg| dx
   \\
   \le C_3 \bigg(\gamma_1 \int_{\Omega_\varepsilon}|\partial_{x_1}U_\varepsilon|^2 dx + \frac{1}{\gamma_1}\int_{\Omega_\varepsilon} W^2_\varepsilon\, dx + \int_{\Omega_\varepsilon}|W_\varepsilon|^2 dx + \varepsilon^2 \int_{\Omega_\varepsilon}|\nabla_{\overline{x}_1}U_\varepsilon|^2 dx \bigg)
\end{multline*}
for any $ \gamma_1>0,$
while the second one based on \eqref{est-F} is less than
$$
C_1 \varepsilon^{2 - \beta} \frac{1}{2} \int_{\Omega_\varepsilon}\Big(1+ U^2_\varepsilon \Big) dx \le
C_4 \bigg(\varepsilon^{4 - \beta} +  \varepsilon^{2 - \beta}\int_{\Omega_\varepsilon}|\partial_{x_1} U_\varepsilon|^2 \, dx\bigg) .
$$
By choosing an appropriate value of $\gamma_1$ and $\varepsilon_1,$  we derive from \eqref{Uin2-be}  the  inequality
\begin{equation}\label{U-in3-be}
  \int_{\Omega_\varepsilon} |\partial_{x_1} U_\varepsilon|^2\, dx + \varepsilon^\beta \, \int_{\Omega_\varepsilon}
|\nabla_{\overline{x}_1} U_\varepsilon|^2\, dx \le C_1 \bigg( \varepsilon^{4 - \beta} +  \int_{\Omega_\varepsilon} W^2_\varepsilon\, dx\bigg)
\end{equation}
for any $\varepsilon \in (0, \varepsilon_1).$

Now the analogue of inequality \eqref{ine-W-e}  looks like this:
\begin{multline}\label{ine-W-e-be+}
  \int_{\Omega_\varepsilon} W^2_\varepsilon(x,\tau)\, dx +
     \int_{0}^{\tau} \int_{\Omega_\varepsilon}\Big(
     |\partial_{x_1} W_\varepsilon|^2 + \varepsilon^\beta |\nabla_{\overline{x}_1} W_\varepsilon|^2\Big) dx\, dt
     \\
    \le C_0 \bigg(|I_1(\varepsilon)| + |I_2(\varepsilon)| + \varepsilon^{2-\beta} \int_{0}^{\tau}\int_{\Omega_\varepsilon} \mathcal{F}^{(2)}_\varepsilon \, W_\varepsilon\, dx\, dt\bigg),
\end{multline}
where $I_1$ and $I_2$ are determined by \eqref{I-1} and \eqref{I-2} respectively, $\tau \in (0,T).$  We will proceed to evaluate the integrals on the right-hand side of \eqref{ine-W-e-be+} as follows:
\begin{align*}
C_0\,  |I_1(\varepsilon)|&\le C_1 \int_{0}^{\tau} \int_{\Omega_\varepsilon}|W_\varepsilon|\, \Big(|\partial_{x_1} W_\varepsilon| + \varepsilon
  |\nabla_{\overline{x}_1} W_\varepsilon|\Big) dx\, dt
  \\
   & \le \gamma_2 \int_{0}^{\tau}\int_{\Omega_\varepsilon}|\partial_{x_1}W_\varepsilon|^2 dx\,dt + \frac{C_2}{\gamma_2}\int_{0}^{\tau}\int_{\Omega_\varepsilon} W^2_\varepsilon\, dx\,dt
   \\ & \  \ \ + C_3 \int_{0}^{\tau}\int_{\Omega_\varepsilon} W^2_\varepsilon\, dx\,dt
   + C_3 \, \varepsilon^2 \int_{0}^{\tau}\int_{\Omega_\varepsilon} |\nabla_{\overline{x}_1} W_\varepsilon|^2 \, dx\,dt \quad \text{for any} \ \ \gamma_2>0,
\end{align*}
\begin{align*}
C_0\,  |I_2(\varepsilon)|&\le C_4 \int_{0}^{\tau} \int_{\Omega_\varepsilon}\Big(|\partial_{x_1}U_\varepsilon|\, |\partial_{x_1}W_\varepsilon|
  + \varepsilon^\beta |\nabla_{\overline{x}_1} U_\varepsilon|\,  |\nabla_{\overline{x}_1} W_\varepsilon|\Big) dx\,dt
  \\
  & \le \gamma_3 \int_{0}^{\tau}\int_{\Omega_\varepsilon}|\partial_{x_1}W_\varepsilon|^2 dx\,dt + \frac{C_5}{\gamma_3}\int_{0}^{\tau}\int_{\Omega_\varepsilon} |\partial_{x_1}U_\varepsilon|^2 \, dx\,dt
  \\ &
 \ \ \ + \varepsilon^\beta \gamma_4 \int_{0}^{\tau}\int_{\Omega_\varepsilon} |\nabla_{\overline{x}_1} W_\varepsilon|^2 dx\, dt +
\varepsilon^\beta\,  \frac{C_6}{\gamma_4} \int_{0}^{\tau}\int_{\Omega_\varepsilon} |\nabla_{\overline{x}_1} U_\varepsilon|^2 dx\, dt
\end{align*}
for any  $\gamma_3>0, \ \ \gamma_4 >0,$
$$
C_0\, \bigg|\varepsilon^{2-\beta} \int_{0}^{\tau}\int_{\Omega_\varepsilon} \mathcal{F}^{(2)}_\varepsilon \, W_\varepsilon\, dx\, dt\bigg|
\stackrel{\eqref{est-F}}{\le}
C_7\bigg(\varepsilon^{6 - 2\beta} + \int_{0}^{\tau}\int_{\Omega_\varepsilon} W^2_\varepsilon\, dx\,dt\bigg).
$$
Choosing the corresponding $\gamma_2, \gamma_3, \gamma_4$ and $\varepsilon_0 > 0$ $(\varepsilon_0 < \varepsilon_1),$  we derive from \eqref{ine-W-e-be+} that for arbitrary $\varepsilon \in (0, \varepsilon_0)$ the following estimate holds:
\begin{multline}\label{est-W-e-be}
  \int_{\Omega_\varepsilon} W^2_\varepsilon(x,\tau)\, dx +
     \int_{0}^{\tau} \int_{\Omega_\varepsilon}\Big(
     |\partial_{x_1} W_\varepsilon|^2 + \varepsilon^\beta |\nabla_{\overline{x}_1} W_\varepsilon|^2\Big) dx\, dt
     \\
    \le C_0 \bigg(\int_{0}^{\tau}\int_{\Omega_\varepsilon} W^2_\varepsilon\, dx\,dt + \varepsilon^\beta \int_{0}^{\tau}\int_{\Omega_\varepsilon} |\nabla_{\overline{x}_1} U_\varepsilon|^2 dx\, dt + \varepsilon^{6 - 2\beta}\bigg)
    \\ \stackrel{\eqref{U-in3-be}}{\le} C_1 \bigg( \varepsilon^{4 - \beta} + \int_{0}^{\tau}\int_{\Omega_\varepsilon} W^2_\varepsilon\, dx\,dt\bigg).
\end{multline}

The remaining step is to apply Gronwall's  lemma, as presented at the end of Theorem~\ref{Th_1}, in order to derive  estimates \eqref{app-estimate-U-e-be} -- \eqref{app-estimateW-ep2-be}.
\end{proof}

Similarly, the proof of Corollary~\ref{Corollary3-1}  may also be applied to obtain the following corollary;
however, now the asymptotic estimates from Theorem~\ref{Th_1+} should be applied.
\begin{corollary}[the case $\alpha =1,$ $\beta < 2$ and $\beta \neq  0$]\label{Corollary4-1}
 Let $(P_\varepsilon, S_\varepsilon)$ be a weak solution to problem $(\mathbb{P}_\varepsilon\mathbb{S}_\varepsilon\!),$ and
 $(p_0, s_0)$ be a solution to the limit problem \eqref{limit_prob}.  Then the following asymptotic estimates hold:
 \begin{gather}\label{est-P-p}
   \max_{t\in[0,T]}{\| \langle {P}_\varepsilon \rangle_{\varepsilon\varpi} - p_0\|}_{H^1(0, \ell)} + {\| \langle {P}_\varepsilon \rangle_{\varepsilon\varpi} - p_0\|}_{C\left([0, \ell]\times[0,T]\right)} \le C_1 \, \varepsilon^{1 - \frac{\beta}{2}},
   \\[2pt]
   {\| \langle {S}_\varepsilon \rangle_{\varepsilon\varpi} - s_0\|}_{L^2(0,T; H^1(0, \ell))} + {\| \langle {S}_\varepsilon \rangle_{\varepsilon\varpi} - s_0\|}_{L^2(0,T; C([0, \ell]))} \le C_2 \, \varepsilon^{1 - \frac{\beta}{2}}. \label{est-S-s}
   \end{gather}
 \end{corollary}

\begin{remark}\label{remark-4-1}
The proofs of the estimates in Theorem~\ref{Th_1+} and in Corollary~\ref{Corollary4-1}  also include the case where $\beta=0.$ However, the order of these estimates is worse than those obtained in Theorem~\ref{Th_1} and Corollary~\ref{Corollary3-1}. Therefore, the subcase $\alpha =1$ and $\beta =0$ was  considered separately. In order to obtain better estimates in Theorem~\ref{Th_1+}, especially for value of  $\beta \in (1,2),$ additional terms of the asymptotics should be constructed. However, for nonlinear problems, as we have shown in \cite{Mel-Roh_JMAA-2024}, this requires additional cumbersome calculations, and for some problems it is almost impossible. In our case, to justify  one-dimensional models, as follows from estimates \eqref{est-P-p}  and \eqref{est-S-s}, this is sufficient.
 \end{remark}

%%%%%%%%%%%%%%%%%%%%

\section{The case $\alpha > \beta -1,$ $\alpha > 1$}\label{Sect:5}

In this case we also consider two subcase.

\subsection{Subcase $\beta \neq  0$}
Now  we  propose the following approximation functions:
\begin{equation}\label{Anz-P-al-be}
 \mathfrak{P}_\varepsilon(x,t) = p_0(x_1,t) + \varepsilon^{\alpha  -1} p_{\alpha  -1}(x_1,t)  + \varepsilon^{\alpha - \beta +1} u_{\alpha- \beta+1}\Big(x_1, \dfrac{\overline{x}_1}{\varepsilon}, t \Big)
\end{equation}
and
\begin{equation}\label{Anz-S-al-be}
 \mathfrak{S}_\varepsilon(x,t) = s_0(x_1,t) +  \varepsilon^{\alpha  -1} s_{\alpha  -1}(x_1,t),
\end{equation}
where $(p_{0}, s_{0})$ is a solution to the problem
\begin{equation}\label{limit_prob-p0-s0-el}
 \left\{\begin{array}{l}
\partial_{x_1}\big(\lambda(s_0)\,  k_1(x_1) \, \partial_{x_1}p_0\big) =  0\quad \text{in} \ \ \mathcal{I}^T,
\\[3pt]
 \phi \, \partial_t s_0 = \partial_{x_1}\big(\Lambda(s_0) \, k_1 \, \partial_{x_1}s_0 + \lambda_w(s_0) \, k_1 \, \partial_{x_1}p_0\big) \quad \text{in} \ \ \mathcal{I}^T,
 \\[4pt]
 p_0(0,t) = q_0(t) \quad \text{and} \quad  p_0(\ell,t) = q_\ell(t), \quad t\in [0,T],
 \\[4pt]
 s_0(0,t)  = S^0(0,t) \quad \text{and} \quad  s_0(\ell,t) = S^0(\ell,t), \quad t\in [0,T],
 \\[3pt]
 s_0(x_1,0) = S^0(x_1,0), \quad x_1 \in [0, \ell],
 \end{array}\right.
\end{equation}
$(p_{\alpha  -1}, s_{\alpha  -1})$ is a solution to the problem
\begin{equation}\label{limit_prob-al-be}
 \left\{\begin{array}{l}
\partial_{x_1}\big(\lambda(s_0)\,  k_1(x_1) \, \partial_{x_1}p_{\alpha  -1} + \lambda'(s_0)\,  k_1(x_1) \, \partial_{x_1}p_0 \, s_{\alpha  -1} \big) =  \widehat{Q}(x_1,t)\quad \text{in} \ \ \mathcal{I}^T,
\\[3pt]
 \phi \, \partial_t s_{\alpha  -1} = \partial_{x_1}\big(\Lambda(s_0) \, k_1 \, \partial_{x_1}s_{\alpha  -1} +   a_1\, s_{\alpha  -1}  +  a_2 \, \partial_{x_1}p_{\alpha  -1}\big) -  b(s_0)\, \widehat{Q}\quad \text{in} \ \ \mathcal{I}^T,
 \\[4pt]
 p_{\alpha  -1}(0,t) = 0 \quad \text{and} \quad  p_{\alpha  -1}(\ell,t) = 0, \quad t\in [0,T],
 \\[4pt]
 s_{\alpha  -1}(0,t)  = 0 \quad \text{and} \quad  s_{\alpha  -1}(\ell,t) = 0, \quad t\in [0,T],
 \\[3pt]
 s_{\alpha  -1}(x_1,0) = 0, \quad x_1 \in [0, \ell],
 \end{array}\right.
\end{equation}
and $u_{\alpha- \beta+1}$ is a solution to problem \eqref{Neumann u2}.

As in \S~\ref{subsect-3-2}, we show that  problem \eqref{limit_prob-p0-s0-el}
has a unique solution from  the H\"older space  $\mathcal{C}^{2+\gamma, 1 +\gamma}([0,\ell]\times[0,T]).$
It should be noted that the compatibility condition of first order  for problem \eqref{limit_prob-p0-s0-el} coincides with \eqref{com-cond-1} since
$Q|_{t=0}=0.$

The coefficients $a_1$ and $a_2$ in the linear problem \eqref{limit_prob-al-be}  are determined by formulas \eqref{coeff-a1-a2} through the solution $(p_0, s_0)$  to  problem \eqref{limit_prob-p0-s0-el}. As in \S~\ref{subsect-3-3} (see Remark~\ref{remark-3-2}), we can show that
problem \eqref{limit_prob-al-be} has a unique classical solution.

Substituting  approximations \eqref{Anz-P-al-be} and \eqref{Anz-S-al-be} into problem $(\mathbb{P}_\varepsilon\mathbb{S}_\varepsilon\!),$  we find
\begin{gather}\label{eq-P-e-al-be}
\nabla \boldsymbol{\cdot} \big(\lambda(\mathfrak{S}_\varepsilon) \,  \mathbb{K}_\varepsilon(x)  \nabla \mathfrak{P}_\varepsilon \big) = \varepsilon^{2\alpha -2} \mathcal{F}^{(1)}_{1,\varepsilon} + \varepsilon^{\alpha - \beta +1} \mathcal{F}^{(1)}_{2,\varepsilon}
  \quad \text{in} \ \ \Omega_\varepsilon^T,
  \\ \label{eq-S-e-al-be}
 \phi(x_1) \, \partial_t  \mathfrak{S}_\varepsilon =  \nabla \boldsymbol{\cdot} \big(\Lambda(\mathfrak{S}_\varepsilon) \,  \mathbb{K}_\varepsilon(x) \nabla \mathfrak{S}_\varepsilon  -  b(\mathfrak{S}_\varepsilon) \, \overrightarrow{\mathfrak{V}}_\varepsilon \big) + \varepsilon^{2\alpha -2}\mathcal{F}^{(2)}_{1,\varepsilon}  + \varepsilon^{\alpha - \beta +1} \mathcal{F}^{(2)}_{2,\varepsilon} \quad \text{in} \ \ \Omega_\varepsilon^T,
\end{gather}
\begin{gather}
\label{Nbc-P-e-al-be}
\Big( \lambda(\mathfrak{S}_\varepsilon) \,  \mathbb{K}_\varepsilon(x)  \nabla \mathfrak{P}_\varepsilon\Big)  \boldsymbol{\cdot} {\boldsymbol{\nu}}_\varepsilon + \varepsilon^\alpha   Q_\varepsilon(x,t) = \varepsilon^{2\alpha -1} \, {\Phi}^{(1)}_\varepsilon \quad \text{on} \ \  \Gamma^T_\varepsilon,
\\ \label{Nbc-S-e-al-be}
\Big(\Lambda(\mathfrak{S}_\varepsilon) \,  \mathbb{K}_\varepsilon(x) \nabla \mathfrak{S}_\varepsilon  -  b(\mathfrak{S}_\varepsilon) \, \overrightarrow{\mathfrak{V}}_\varepsilon\Big) \boldsymbol{\cdot} {\boldsymbol{\nu}}_\varepsilon + \varepsilon^\alpha b(\mathfrak{S}_\varepsilon) \,
Q_\varepsilon= \varepsilon^{2\alpha -1} \, \Phi^{(2)}_\varepsilon \quad \text{on} \ \  \Gamma^T_\varepsilon,
\end{gather}
\begin{gather}
  \label{bc-P-e-al-be}
  \mathfrak{P}_\varepsilon|_{(x,t)\in \Upsilon^T_\varepsilon(0)} = q_0(t) \quad \text{and} \quad  \mathfrak{P}_\varepsilon|_{(x,t)\in \Upsilon^T_\varepsilon(\ell)} = q_\ell(t),
\\
\label{bc-S-e-al-be}
  \mathfrak{S}_\varepsilon|_{(x,t)\in \Upsilon^T_\varepsilon(0)} = S^0(0,t), \quad   \mathfrak{S}_\varepsilon|_{(x,t)\in \Upsilon^T_\varepsilon(\ell)} = S^0(\ell,t), \quad \mathfrak{S}_\varepsilon|_{t=0} = S^0(x_1,0).
  \end{gather}
Here $ \overrightarrow{\mathfrak{V}}_\varepsilon = -  \lambda(\mathfrak{S}_\varepsilon) \,  \mathbb{K}_\varepsilon(x)  \nabla \mathfrak{P}_\varepsilon$ and
\begin{gather}\label{est-F-al-be}
  \sup_{\Omega_\varepsilon^T}|\mathcal{F}^{(1)}_{i,\varepsilon}(x,t)| \le C_1, \quad
  \sup_{\Omega_\varepsilon^T}|\mathcal{F}^{(2)}_{i,\varepsilon}(x,t)| \le C_2, \quad i\in \{1, 2\},
\\
\label{est-Fhi-al-be}
  \sup_{\Gamma_\varepsilon^T}|\Phi^{(1)}_\varepsilon(x,t)| \le C_3, \quad
  \sup_{\Gamma_\varepsilon^T}|\Phi^{(2)}_\varepsilon(x,t)| \le C_4.
\end{gather}

\begin{theorem}[the case $\alpha > \beta -1,$ $\alpha > 1$ and $\beta \neq  0$]\label{Th_5-1}
Assume that, in addition to the main assumptions made in Section~\ref{Sec:Statement},
condition \eqref{com-cond-1} is  satisfied.
Then,  there are positive constants $\tilde{C}_0,$ $\tilde{C}_1,$ $\tilde{C}_2$ and $\varepsilon_0$ such that for all $\varepsilon\in (0, \varepsilon_0)$
the  differences $U_\varepsilon := \mathfrak{P}_\varepsilon - P_\varepsilon$ and $W_\varepsilon := \mathfrak{S}_\varepsilon - {S}_\varepsilon$
 satisfy the estimates
   \begin{equation}\label{app-estimate-U-e-al-be}
 \tfrac{1}{\sqrt{\upharpoonleft  \Omega_\varepsilon \! \upharpoonright_3}}\,  \max_{t\in \times [0,T]}
 \sqrt{\int_{\Omega_\varepsilon} |\partial_{x_1} U_\varepsilon(\cdot,t)|^2\, dx + \varepsilon^\beta \, \int_{\Omega_\varepsilon}
|\nabla_{\overline{x}_1} U_\varepsilon(\cdot,t)|^2\, dx} \le \tilde{C}_0 \, \varepsilon^{\aleph_{\alpha,\beta}},
 \end{equation}
 \begin{equation}\label{app-estimate-W-ep-al-be}
\tfrac{1}{\sqrt{\upharpoonleft  \Omega_\varepsilon \! \upharpoonright_3}}\,  \max_{t\in \times [0,T]}
{\| W_\varepsilon(\cdot,t) \|}_{L^2(\Omega_\varepsilon)} \le \tilde{C}_1 \,  \varepsilon^{\aleph_{\alpha,\beta}},
\end{equation}
and
\begin{equation}\label{app-estimateW-ep2-al-be}
\tfrac{1}{\sqrt{\upharpoonleft  \Omega_\varepsilon \! \upharpoonright_3}}\,
\sqrt{ \int_{0}^{T} \int_{\Omega_\varepsilon} |\partial_{x_1} W_\varepsilon|^2 \, dx dt
+ \varepsilon^\beta \int_{0}^{T} \int_{\Omega_\varepsilon} |\nabla_{\overline{x}_1} W_\varepsilon|^2 \, dx dt} \le
 \tilde{C}_2\, \varepsilon^{\aleph_{\alpha,\beta}},
\end{equation}
where $\aleph_{\alpha,\beta} = \min\{ \alpha-1 , \, \frac{\alpha-\beta +1}{2}\}.$
 \end{theorem}
\begin{proof} Compared to the previous theorems, the proof of this theorem requires some changes. First, we note that since $s_0$ satisfies inequalities \eqref{limit-saturation} and $s_{\alpha-1}$ is bounded and $\alpha >1,$ there exists an $\varepsilon_1$ such that for all $\varepsilon \in (0, \varepsilon_1)$
\begin{equation}\label{saturation2}
     0< \frac{\delta_0}{2}  \le \mathfrak{S}_\varepsilon(x,t) \le \frac{1+ \delta_1}{2}  < 1, \quad \text{for } \ (x,t) \in \Omega_\varepsilon^T.
\end{equation}
where $\delta_0$ and $\delta_1$ are defined in \eqref{saturation1}.

 The  differences
$U_\varepsilon = \mathfrak{P}_\varepsilon - P_\varepsilon$ and $W_\varepsilon = \mathfrak{S}_\varepsilon - {S}_\varepsilon$
 satisfy relations
 \begin{equation}\label{eq-U-e-al-be}
\nabla \boldsymbol{\cdot} \Big(\lambda({S}_\varepsilon(x,t)) \,  \mathbb{K}_\varepsilon\nabla U_\varepsilon  + W_\varepsilon \, \lambda'(\Theta^{(1)}_\varepsilon)\, \mathbb{K}_\varepsilon\nabla\mathfrak{P}_\varepsilon \Big) =
 \varepsilon^{2\alpha -2} \mathcal{F}^{(1)}_{1,\varepsilon} + \varepsilon^{\alpha - \beta +1} \mathcal{F}^{(1)}_{2,\varepsilon} \quad \text{in} \ \ \Omega_\varepsilon^T,
  \end{equation}
where $\Theta^{(1)}_\varepsilon(x,t) = \theta_1 \, \mathfrak{S}_\varepsilon(x,t) + (1-\theta_1)\, {S}_\varepsilon, \  \theta_1 \in (0,1),$
\begin{multline}\label{eq-W-e-al-be}
 \phi(x_1) \, \partial_t  W_\varepsilon
  -  \nabla \boldsymbol{\cdot} \Big(\Lambda({S}_\varepsilon) \,  \mathbb{K}_\varepsilon(x) \nabla W_\varepsilon + W_\varepsilon \, \Lambda'(\Theta^{(2)}_\varepsilon)\, \mathbb{K}_\varepsilon(x)  \nabla\mathfrak{S}_\varepsilon
 \\
 + \lambda_w({S}_\varepsilon) \,  \mathbb{K}_\varepsilon(x) \nabla U_\varepsilon
 + W_\varepsilon \, \lambda'_w(\Theta^{(3)}_\varepsilon)\, \mathbb{K}_\varepsilon(x)  \nabla\mathfrak{P}_\varepsilon
 \Big)
=  \varepsilon^{2\alpha -2} \mathcal{F}^{(2)}_{1,\varepsilon} + \varepsilon^{\alpha - \beta +1} \mathcal{F}^{(2)}_{2,\varepsilon}  \quad \text{in} \ \ \Omega_\varepsilon^T,
\end{multline}
\begin{equation}\label{Nbc-U-e-al-be}
\Big(\lambda({S}_\varepsilon) \,  \mathbb{K}_\varepsilon(x)  \nabla U_\varepsilon +
 W_\varepsilon \, \lambda'(\Theta^{(1)}_\varepsilon)\, \mathbb{K}_\varepsilon(x)  \nabla\mathfrak{P}_\varepsilon
\Big) \boldsymbol{\cdot} {\boldsymbol{\nu}}_\varepsilon
  =  \varepsilon^{2\alpha -1} \, {\Phi}^{(1)}_\varepsilon \quad \text{on} \ \  \Gamma^T_\varepsilon,
\end{equation}
\begin{multline} \label{Nbc-W-e-al-be}
\Big(\Lambda({S}_\varepsilon) \,  \mathbb{K}_\varepsilon(x) \nabla W_\varepsilon
+ W_\varepsilon \, \Lambda'(\Theta^{(2)}_\varepsilon)\, \mathbb{K}_\varepsilon(x)  \nabla\mathfrak{S}_\varepsilon
 + \lambda_w({S}_\varepsilon) \,  \mathbb{K}_\varepsilon(x) \nabla U_\varepsilon
\\
 + W_\varepsilon \, \lambda'_w(\Theta^{(3)}_\varepsilon)\, \mathbb{K}_\varepsilon(x)  \nabla\mathfrak{P}_\varepsilon \Big) \boldsymbol{\cdot} {\boldsymbol{\nu}}_\varepsilon
=  \varepsilon^{2\alpha -1} \, \Phi^{(2)}_\varepsilon \quad \text{on} \ \  \Gamma^T_\varepsilon,
\end{multline}
and \eqref{bc-U-e} and \eqref{bc-W-e}.

The main difference with the previous theorems is the presence of discrepancies in  boundary conditions
\eqref{Nbc-U-e-al-be} and \eqref{Nbc-W-e-al-be}.
Multiplying the corresponding equation \eqref{eq-U-e-al-be} by $U_\varepsilon$ and integrating it over $\Omega_\varepsilon,$ we get
\begin{multline}\label{U-ine1-al-be}
\int_{\Omega_\varepsilon} \lambda({S}_\varepsilon) \,  \mathbb{K}_\varepsilon  \nabla U_\varepsilon \boldsymbol{\cdot} \nabla U_\varepsilon \, dx = - \int_{\Omega_\varepsilon} W_\varepsilon \, \lambda'(\Theta^{(1)}_\varepsilon)\, \mathbb{K}_\varepsilon  \nabla\mathfrak{P}_\varepsilon \boldsymbol{\cdot} \nabla U_\varepsilon \, dx
   \\
   + \varepsilon^{2\alpha -1} \int_{\Gamma_\varepsilon} {\Phi}^{(1)}_\varepsilon\, U_\varepsilon\, d\sigma_x  -  \varepsilon^{2\alpha -2} \int_{\Omega_\varepsilon} \mathcal{F}^{(1)}_{1,\varepsilon}
\, U_\varepsilon\, dx -  \varepsilon^{\alpha - \beta +1} \int_{\Omega_\varepsilon} \mathcal{F}^{(1)}_{2,\varepsilon}
\, U_\varepsilon\, dx.
\end{multline}
Taking into account the form of the  absolute permeability tensor \eqref{permeability} and using Cauchy's inequality with any $\gamma >0,$ the first integral on the right-hand side of \eqref{U-ine1-al-be} is bounded from above by the sum
$$
C_1 \bigg(\gamma \int_{\Omega_\varepsilon} |\partial_{x_1} U_\varepsilon|^2\, dx + \frac{1}{\gamma} \int_{\Omega_\varepsilon}W_\varepsilon^2\, dx + \varepsilon^{\alpha -1} \int_{\Omega_\varepsilon}W_\varepsilon^2\, dx + \varepsilon^{\alpha +1} \int_{\Omega_\varepsilon}|\nabla_{\overline{x}_1} U_\varepsilon|^2\, dx\bigg).
$$
Considering \eqref{est-F-al-be}, \eqref{est-Fhi-al-be} and \eqref{bc-U-e}, the remaining integrals on the right-hand side of \eqref{U-ine1-al-be} are estimated as follows
\begin{gather}\label{F-1-al-be}
  \varepsilon^{2\alpha -2} \bigg|\int_{\Omega_\varepsilon} \mathcal{F}^{(1)}_{1,\varepsilon}
\, U_\varepsilon\, dx \bigg| \le C_2 \bigg(\varepsilon^{2\alpha} + \varepsilon^{2\alpha -2} \int_{\Omega_\varepsilon} |\partial_{x_1} U_\varepsilon|^2\, dx\bigg),
   \\ \label{F-2-al-be}
  \varepsilon^{\alpha - \beta +1} \bigg|\int_{\Omega_\varepsilon} \mathcal{F}^{(1)}_{2,\varepsilon}
\, U_\varepsilon\, dx\bigg| \le C_3 \bigg(\varepsilon^{\alpha - \beta +3} + \varepsilon^{\alpha - \beta +1} \int_{\Omega_\varepsilon} |\partial_{x_1} U_\varepsilon|^2\, dx\bigg),
\end{gather}
\begin{equation}\label{Fi-1-al-be}
  \varepsilon^{2\alpha -1} \bigg|\int_{\Gamma_\varepsilon} {\Phi}^{(1)}_\varepsilon\, U_\varepsilon\, d\sigma_x\bigg| \stackrel{\eqref{ineq1}}{\le} C_4\bigg(\varepsilon^{2\alpha} + \varepsilon^{2\alpha}\int_{\Omega_\varepsilon}|\nabla_{\overline{x}_1} U_\varepsilon|^2\, dx +      \varepsilon^{2\alpha-2} \int_{\Omega_\varepsilon} |\partial_{x_1} U_\varepsilon|^2\, dx\bigg).
\end{equation}
As a result, based on these estimates and taking suitable values for $\gamma$ and $\varepsilon_0,$ we deduce from \eqref{U-ine1-al-be} the inequality
\begin{equation}\label{U-ine2-al-be}
\int_{\Omega_\varepsilon} |\partial_{x_1} U_\varepsilon|^2\, dx + \varepsilon^\beta \, \int_{\Omega_\varepsilon}
|\nabla_{\overline{x}_1} U_\varepsilon|^2\, dx \le C_5 \Big(\int_{\Omega_\varepsilon}W_\varepsilon^2\, dx + \varepsilon^{2\alpha } +\varepsilon^{\alpha - \beta +3}\Big)
\end{equation}
for any $\varepsilon \in (0, \varepsilon_0).$

Taking into account \eqref{permeability}, \eqref{cercitive} and \eqref{total mobillity},   inequality \eqref{ine-W-e-be+} now looks like
\begin{multline}\label{ine-W-e-al-be}
  \int_{\Omega_\varepsilon} W^2_\varepsilon(x,\tau)\, dx + \int_{0}^{\tau} \int_{\Omega_\varepsilon} |\partial_{x_1} W_\varepsilon|^2 \, dx\, dt
+ \varepsilon^\beta \int_{0}^{\tau} \int_{\Omega_\varepsilon} |\nabla_{\overline{x}_1} W_\varepsilon|^2 \, dx\, dt
\\
  \le C_0 \bigg(|I_1(\varepsilon)| +| I_2(\varepsilon)|
+ \varepsilon^{2\alpha -1} \int_{0}^{\tau}\int_{\Gamma_\varepsilon} |{\Phi}^{(2)}_\varepsilon|\, |W_\varepsilon|\, d\sigma_x dt
\\
 +  \varepsilon^{2\alpha -2} \int_{0}^{\tau}\int_{\Omega_\varepsilon} |\mathcal{F}^{(2)}_{1,\varepsilon}|
\, |W_\varepsilon|\, dx dt +  \varepsilon^{\alpha-\beta+1} \int_{0}^{\tau}\int_{\Omega_\varepsilon} |\mathcal{F}^{(2)}_{2,\varepsilon}|
\, |W_\varepsilon|\, dx dt
\bigg),
\end{multline}
where $I_1(\varepsilon)$ and $I_2(\varepsilon)$ are determined by formulas \eqref{I-1} and \eqref{I-2}, respectively.
The integrals in the right-hand side  of \eqref{ine-W-e-al-be} are estimated in the same way as integrals \eqref{F-1-al-be} -- \eqref{Fi-1-al-be}.
Using  Cauchy's inequality with any $\gamma >0,$ we bound $I_1(\varepsilon)$ and $I_2(\varepsilon):$
\begin{align*}
 |I_1(\varepsilon)|& \le  \gamma_1 \int_{0}^{\tau}\int_{\Omega_\varepsilon} |\partial_{x_1} W_\varepsilon|^2\, dx dt + \frac{C_1}{\gamma_1} \int_{0}^{\tau}\int_{\Omega_\varepsilon} |W_\varepsilon|^2\, dx dt
 \\
 & \  \ \ +
 \varepsilon^{\alpha-1} C_2 \int_{0}^{\tau}\int_{\Omega_\varepsilon} | W_\varepsilon|^2\, dx dt + \varepsilon^{\alpha+1} C_3 \, \int_{0}^{\tau}\int_{\Omega_\varepsilon} |\nabla_{\overline{x}_1} W_\varepsilon|^2\, dx dt,
\end{align*}
\begin{align*}
 |I_2(\varepsilon)|& = \bigg|\int_{0}^{\tau} \int_{\Omega_\varepsilon} \lambda_w(S_\varepsilon)\, k_1(x_1)\, \partial_{x_1}U_\varepsilon\, \partial_{x_1}W_\varepsilon\,dx dt
 \\ &
\ \ \ \  +
 \int_{0}^{\tau} \int_{\Omega_\varepsilon}\lambda_w(S_\varepsilon)\,  \varepsilon^\beta \, (K\nabla_{\overline{x}_1} U_\varepsilon)\cdot \nabla_{\overline{x}_1} W_\varepsilon dx dt\bigg|
   \\
   &
  \le \gamma_2 \int_{0}^{\tau}\int_{\Omega_\varepsilon} |\partial_{x_1} W_\varepsilon|^2\, dx dt + \frac{C_4}{\gamma_2} \int_{0}^{\tau}\int_{\Omega_\varepsilon} |\partial_{x_1}U_\varepsilon|^2\, dx dt
  \\
 & \ \ \ +
 \varepsilon^\beta \, \gamma_3 \int_{0}^{\tau}\int_{\Omega_\varepsilon} |\nabla_{\overline{x}_1} W_\varepsilon|^2\, dx dt + \varepsilon^\beta \, \frac{C_5}{\gamma_3} \int_{0}^{\tau}\int_{\Omega_\varepsilon} |\nabla_{\overline{x}_1} U_\varepsilon|^2\, dx dt.
\end{align*}

By choosing $\gamma_1,$ $\gamma_2$, $\gamma_3$ and $\varepsilon_1$ to be suitably small, we derive from \eqref{ine-W-e-al-be} the inequality
\begin{multline}\label{ine-W-e-al-be+1}
  \int_{\Omega_\varepsilon} W^2_\varepsilon(x,\tau)\, dx + \int_{0}^{\tau} \int_{\Omega_\varepsilon} |\partial_{x_1} W_\varepsilon|^2 \, dx\, dt
+ \varepsilon^\beta \int_{0}^{\tau} \int_{\Omega_\varepsilon} |\nabla_{\overline{x}_1} W_\varepsilon|^2 \, dx\, dt
\\
\le C_3 \bigg(\varepsilon^{2\alpha } +\varepsilon^{\alpha - \beta +3}  + \int_{0}^{\tau}\int_{\Omega_\varepsilon} | W_\varepsilon|^2\, dx dt +
\int_{0}^{\tau}\int_{\Omega_\varepsilon} |\partial_{x_1} U_\varepsilon|^2\, dx dt
 + \varepsilon^\beta \, \int_{0}^{\tau}\int_{\Omega_\varepsilon}
|\nabla_{\overline{x}_1} U_\varepsilon|^2\, dx dt\bigg)
\\
\stackrel{\eqref{U-ine2-al-be}}{\le} C_4\bigg(\varepsilon^{2\alpha } +\varepsilon^{\alpha - \beta +3}  + \int_{0}^{\tau}\int_{\Omega_\varepsilon} | W_\varepsilon|^2\, dx dt\bigg)
\end{multline}
for any $\varepsilon \in (0, \varepsilon_1).$

Now applying Gronwall's lemma's to the inequality
$$
\int_{\Omega_\varepsilon} W^2_\varepsilon(x,\tau)\, dx  \le C_4 \Big(\varepsilon^{2\alpha } +\varepsilon^{\alpha - \beta +3}  + \int_{0}^{\tau}\int_{\Omega_\varepsilon} | W_\varepsilon|^2\, dx dt\Big),
$$
we get
\begin{equation}\label{uniform-est-W-e-al-be}
  \max_{t\in [0,T]} \int_{\Omega_\varepsilon} W^2_\varepsilon(x,t)\, dx \le C_5 \, \big(\varepsilon^{2\alpha } +\varepsilon^{\alpha - \beta +3}\big).
\end{equation}

It is easy to see that estimates \eqref{app-estimate-U-e-al-be} -- \eqref{app-estimateW-ep2-al-be} follow directly from  \eqref{uniform-est-W-e-al-be}, \eqref{ine-W-e-al-be+1} and \eqref{U-ine2-al-be}.
\end{proof}

The next corollary is a consequence  of estimates \eqref{app-estimate-U-e-al-be} -- \eqref{app-estimateW-ep2-al-be}.

\begin{corollary}[the case $\alpha > \beta -1,$ $\alpha > 1$ and $\beta \neq  0$]\label{Corollary5-1}
 Let $(P_\varepsilon, S_\varepsilon)$ be a weak solution to problem $(\mathbb{P}_\varepsilon\mathbb{S}_\varepsilon\!),$ and
 $(p_0, s_0)$ be a solution to  problem \eqref{limit_prob-p0-s0-el}.  Then the following asymptotic estimates hold:
 \begin{equation}\label{est-P-p-al-be}
   \max_{t\in[0,T]}{\| \langle {P}_\varepsilon \rangle_{\varepsilon\varpi} - p_0 - \varepsilon^{\alpha -1} p_{\alpha -1}\|}_{H^1(0, \ell)}
        + {\| \langle {P}_\varepsilon \rangle_{\varepsilon\varpi} - p_0 - \varepsilon^{\alpha -1} p_{\alpha -1}\|}_{C\left([0, \ell]\times[0,T]\right)} \le C_1 \, \varepsilon^{\aleph_{\alpha,\beta}},
  \end{equation}
 \begin{equation}\label{est-S-s-al-be}
      {\| \langle {S}_\varepsilon \rangle_{\varepsilon\varpi} - s_0 - \varepsilon^{\alpha -1} s_{\alpha -1} \|}_{L^2(0,T; H^1(0, \ell))}
            + {\| \langle {S}_\varepsilon \rangle_{\varepsilon\varpi} - s_0 - \varepsilon^{\alpha -1} s_{\alpha -1}\|}_{L^2(0,T; C([0, \ell]))} \le C_2 \, \varepsilon^{\aleph_{\alpha,\beta}},
 \end{equation}
 where $\aleph_{\alpha,\beta} = \min\{ \alpha-1 , \, \frac{\alpha-\beta +1}{2}\}.$
 \end{corollary}

%%%%%%%%%%%%%%%%%%%%%
\subsection{Subcase $\beta =0$}

This subcase is considered as it provides an opportunity to obtain more precise estimates than those presented in Theorem~\ref{Th_5-1} and Corollary~\ref{Corollary5-1}.
It is obvious that now  the following approximation functions should be considered:
\begin{equation}\label{Anz-P-al}
 \mathfrak{P}_\varepsilon(x,t) = p_0(x_1,t) + \varepsilon^{\alpha  -1} p_{\alpha  -1}(x_1,t)  + \varepsilon^{\alpha  +1} u_{\alpha+1}\Big(x_1, \dfrac{\overline{x}_1}{\varepsilon}, t \Big)
\end{equation}
and
\begin{equation}\label{Anz-S-al}
 \mathfrak{S}_\varepsilon(x,t) = s_0(x_1,t) +  \varepsilon^{\alpha  -1} s_{\alpha  -1}(x_1,t),
\end{equation}
where
$(p_0, s_0)$ and $(p_{\alpha-1}, s_{\alpha-1})$ are  solutions to problems \eqref{limit_prob-p0-s0-el} and \eqref{limit_prob-al-be}, respectively,
$u_{\alpha+1}$
is a solution to the problem \eqref{Neumann u2}.

\begin{theorem}[the case $\alpha > 1, \, \beta =0$]\label{Th_2}
Assume that, in addition to the main assumptions made in Section~\ref{Sec:Statement},
condition \eqref{com-cond-1} is  satisfied.
Then,  there are positive constants $\tilde{C}_0,$ $\tilde{C}_1$ and $\varepsilon_0$ such that for all $\varepsilon\in (0, \varepsilon_0)$
 \begin{equation}\label{app-estimate1-al}
 \tfrac{1}{\sqrt{\upharpoonleft  \Omega_\varepsilon \! \upharpoonright_3}}\,  \max_{t\in \times [0,T]}
 {\|P_\varepsilon(\cdot,t) - \mathfrak{P}_\varepsilon(\cdot,t) \|}_{H^1(\Omega_\varepsilon)}  \le \tilde{C}_0 \,  \varepsilon^{\aleph_\alpha}
\end{equation}
and
\begin{equation}\label{app-estimate2-al}
\tfrac{1}{\sqrt{\upharpoonleft  \Omega_\varepsilon \! \upharpoonright_3}}\,  \max_{t\in \times [0,T]}
{\| S_\varepsilon(\cdot,t) - \mathfrak{S}_\varepsilon(\cdot,t) \|}_{L^2(\Omega_\varepsilon)}
+
\tfrac{1}{\sqrt{\upharpoonleft  \Omega_\varepsilon \! \upharpoonright_3}}\,
{\|S_\varepsilon - \mathfrak{S}_\varepsilon\|}_{L^2(0,T; H^1(\Omega_\varepsilon))} \le
 \tilde{C}_1\, \varepsilon^{\aleph_\alpha},
\end{equation}
where $\aleph_\alpha := \min\{2(\alpha-1), \, \alpha +1\}.$
 \end{theorem}
\begin{proof}
The proof of the theorem requires some clarifications compared to the proof of Theorem~\ref{Th_1}.
Substituting   functions \eqref{Anz-P-al} and \eqref{Anz-S-al}  into problem $(\mathbb{P}_\varepsilon\mathbb{S}_\varepsilon\!),$  we find the same  relations  as  \eqref{eq-P-e-al-be} -- \eqref{bc-S-e-al-be}, but with $\beta =0.$ The same for the differences
$U_\varepsilon = \mathfrak{P}_\varepsilon - P_\varepsilon$ and $W_\varepsilon = \mathfrak{S}_\varepsilon - {S}_\varepsilon;$ they
satisfy relations \eqref{eq-U-e-al-be} -- \eqref{Nbc-W-e-al-be} in which $\beta =0.$

The main difference with the corresponding stage in the proof of Theorem~\ref{Th_1} is the presence of discrepancies in  boundary conditions
\eqref{Nbc-U-e} and \eqref{Nbc-W-e}. Therefore, taking into account \eqref{est-Fhi-al-be}, the integral
$
\varepsilon^{2\alpha -1} \int_{\Gamma_\varepsilon} |U_\varepsilon|\, d\sigma_x
$
will appear on the right-hand side of inequality \eqref{in-U-e}. Using  \eqref{ineq1} and \eqref{bc-U-e}, we get
$$
\varepsilon^{2\alpha -1} \int_{\Gamma_\varepsilon} |U_\varepsilon|\, d\sigma_x \le
C\, \varepsilon^{2\alpha -1} \sqrt{\int_{\Omega_\varepsilon}|\nabla U_\varepsilon|^2\, dx} .
$$

 As a result,
\begin{equation}\label{in-U-e+1+al}
  \int_{\Omega_\varepsilon} |\nabla U_\varepsilon(x,t)|^2\, dx \le C_1 \Big( \varepsilon^{4\alpha -2} + \varepsilon^{\alpha +2} + \int_{\Omega_\varepsilon} W^2_\varepsilon(x,t)  dx\Big)
  \quad \text{for any} \ \ t \in [0,T].
\end{equation}

Also on the right side of inequality \eqref{ine-W-e} there will be an additional integral, which is estimated as follows
\begin{multline}\label{est-Phi-2}
\bigg|\varepsilon^{2\alpha -1} \int_{0}^{\tau}\int_{\Gamma_\varepsilon} {\Phi}^{(2)}_\varepsilon
\, W_\varepsilon\, d\sigma_x\, dt \bigg|
\stackrel{\eqref{est-Fhi-al-be}}{\le}  C_2 \Big(\varepsilon^{4\alpha -2} + \varepsilon \int_{0}^{\tau}\int_{\Gamma_\varepsilon} W^2_\varepsilon \, d\sigma_x dt\Big)
\\
\stackrel{\eqref{ineq1}}{\le}  C_3 \Big(\varepsilon^{4\alpha -2} + \varepsilon^2 \int_{0}^{\tau}\int_{\Omega_\varepsilon}|\nabla_{\overline{x}_1} W_\varepsilon|^2 dx dt + \int_{0}^{\tau}\int_{\Omega_\varepsilon} W_\varepsilon^2 dx dt\Big).
\end{multline}
We then choose $\varepsilon_0$ such that for all $\varepsilon\in (0, \varepsilon_0)$  the value  $C_0 \, C_3\, \varepsilon^{2}$ is less than a given positive number, so that we can move the integral
$
C_0 \, C_3\, \varepsilon^{2} \, \int_{0}^{\tau}\int_{\Omega_\varepsilon}|\nabla_{\overline{x}_1} W_\varepsilon|^2 dx dt
$
to  the left-hand side of the corresponding inequality \eqref{ine-W-e} and obtain the inequality
\begin{equation}\label{ine-W-e+2-al}
  \int_{\Omega_\varepsilon} W^2_\varepsilon(x,\tau)\, dx + \int_{\Omega^\tau_\varepsilon} |\nabla W_\varepsilon|^2 \, dx dt
    \le C_4 \Big(\varepsilon^{4\alpha -2} + \varepsilon^{\alpha +2} \int_{\Omega^\tau_\varepsilon} W^2_\varepsilon(x,t) \, dx dt\Big)
\end{equation}
for any  $\tau\in (0,T).$ We then repeat the proof of Theorem~\ref{Th_1} to derive estimates \eqref{app-estimate1-al} and \eqref{app-estimate2-al}.
\end{proof}

Just as we proved Corollary~\ref{Corollary3-1}, from inequalities \eqref{app-estimate1-al} and \eqref{app-estimate2-al} we deduce the following corollary.

\begin{corollary}[the case $\alpha >1,  \, \beta =0$]\label{corollary4-1}
 Let $(P_\varepsilon, S_\varepsilon)$ be a weak solution to problem $(\mathbb{P}_\varepsilon\mathbb{S}_\varepsilon\!),$ and
 $(p_0, s_0)$ and $(p_{\alpha-1}, s_{\alpha-1})$ be  solutions to problems \eqref{limit_prob-p0-s0-el} and \eqref{limit_prob-al-be}, respectively.  Then the following estimates hold:
 \begin{equation}\label{est-P-p-al}
   \max_{t\in[0,T]}{\| \langle {P}_\varepsilon \rangle_{\varepsilon\varpi} - p_0 - \varepsilon^{\alpha -1} p_{\alpha -1}\|}_{H^1(0, \ell)}
      + {\| \langle {P}_\varepsilon \rangle_{\varepsilon\varpi} - p_0  - \varepsilon^{\alpha -1} p_{\alpha -1}\|}_{C\left([0, \ell]\times[0,T]\right)} \le C_1 \, \varepsilon^{\aleph_\alpha},
 \end{equation}
 \begin{equation}
   {\| \langle {S}_\varepsilon \rangle_{\varepsilon\varpi} - s_0 - \varepsilon^{\alpha -1} s_{\alpha -1} \|}_{L^2(0,T; H^1(0, \ell))}
      + {\| \langle {S}_\varepsilon \rangle_{\varepsilon\varpi} - s_0 - \varepsilon^{\alpha -1} s_{\alpha -1}\|}_{L^2(0,T; C([0, \ell]))} \le C_2 \, \varepsilon^{\aleph_\alpha}, \label{est-S-s-al}
   \end{equation}
 where $\aleph_\alpha = \min\{2(\alpha-1), \, \alpha +1\}.$
 \end{corollary}

%%%%%%%%%%%%%%%%%%%%%%%%%

 \section{Approximations and estimates for the flow velocities and pressures}\label{Sect-6}

The obtained approximations for the solution to problem $(\mathbb{P}_\varepsilon\mathbb{S}_\varepsilon\!)$ for different values of the parameters $\alpha$ and $\beta$ allow us to return to the Muskat-Leverett two-phase flow model (see equations \eqref{saturation} - \eqref{capilarity}), to construct  approximations for the flow velocities $\overrightarrow{V_\varepsilon},$ $\vec{V}_{w,\varepsilon}$ and $\vec{V}_{o,\varepsilon}$ and for the pressures $P_{w, \varepsilon}$ and   $P_{o, \varepsilon},$ and to prove corresponding asymptotic estimates. As follows from the preceding sections, the best asymptotic estimates are derived  when $\beta = 0.$ Therefore, we will start with this case.

$\blacklozenge$ If $\alpha = 1$ and $\beta =0,$ we propose for the total velocity $\overrightarrow{V_\varepsilon}$  the approximation
\begin{equation}\label{app-total-velocity}
  \overrightarrow{\mathcal{V}_\varepsilon} := -  \left(\begin{matrix}
                                                \lambda(s_0(x_1,t)) \, k_1(x_1) \, \partial_{x_1}p_0(x_1,t)
                                                 \\[4pt]
                                                \varepsilon\, \lambda(s_0(x_1,t)) \, \mathbf{K}(x_1, \frac{\overline{x}_1}{\varepsilon}) \, \big(\nabla_{\overline{\xi}_1} u_2(x_1,\overline{\xi}_1)\big)|_{\overline{\xi}_1 = \frac{\overline{x}_1}{\varepsilon}}
                                              \end{matrix}\right),
\end{equation}
for $\vec{V}_{w,\varepsilon},$ based on \eqref{total3}, the following approximation is offered:
\begin{equation}\label{app-velocity-w}
  \vec{\mathcal{V}}_{w,\varepsilon} := -\left(\begin{matrix}
                                                \Lambda(s_0(x_1,t)) \, k_1(x_1) \, \partial_{x_1}s_0(x_1,t)
                                                 \\ 0 \\ 0
                                              \end{matrix}\right) + b(s_0(x_1,t)) \, \overrightarrow{\mathcal{V}_\varepsilon},
\end{equation}
\begin{equation}\label{app-velocity-o}
  \vec{\mathcal{V}}_{o,\varepsilon} :=  \overrightarrow{\mathcal{V}_\varepsilon} - \vec{\mathcal{V}}_{w,\varepsilon}.
\end{equation}
Here,  the matrix $\mathbf{K}$ is given by \eqref{matrix-K}, $(p_0, s_0)$ is the solution to the limit problem \eqref{limit_parab-prob}, $u_2$ is the solution to problem \eqref{Neumann u2}.

Based on \eqref{mean Pressure} and \eqref{capilarity},  we propose  approximations for the pressures $P_{w, \varepsilon}$ and   $P_{o, \varepsilon}$
in the following form:
\begin{equation}\label{app-pressure-w}
  \mathfrak{P}_{w, \varepsilon} := \mathfrak{P}_{\varepsilon} - \int_{0}^{s_0} \frac{\lambda_o(\eta)}{\lambda(\eta)} \, p_c^\prime(\eta) \,d\eta, \qquad \mathfrak{P}_{o, \varepsilon} := \mathfrak{P}_{w, \varepsilon} + p_c(s_0),
\end{equation}
where $\mathfrak{P}_{\varepsilon}$ is taken from \eqref{AppFunctions}.

\begin{lemma}[$\alpha = 1, \, \beta =0$]\label{lemma-6-1}
  Suppose that the conditions of Theorem~\ref{Th_1} are satisfied. Then, there are positive constants $\tilde{C}_1,$ $\tilde{C}_2,$ $\tilde{C}_3$
   and $\varepsilon_0$ such that for all $\varepsilon\in (0, \varepsilon_0)$
  \begin{equation}\label{app-estimate3}
\tfrac{1}{\sqrt{\upharpoonleft  \Omega_\varepsilon \! \upharpoonright_3}}\,  \max_{t\in \times [0,T]}
{\left\| \overrightarrow{{V}_\varepsilon}(\cdot,t) - \overrightarrow{\mathcal{V}_\varepsilon}(\cdot,t) \right\|}_{(L^2(\Omega_\varepsilon))^3}
 \le
 \tilde{C}_1\, \varepsilon^2,
\end{equation}
\begin{equation}\label{app-estimate4}
\tfrac{1}{\sqrt{\upharpoonleft  \Omega_\varepsilon \! \upharpoonright_3}}\,
{\left\| \vec{V}_{w,\varepsilon} - \vec{\mathcal{V}}_{w,\varepsilon} \right\|}_{(L^2(\Omega^T_\varepsilon)^3}
 \le  \tilde{C}_2\, \varepsilon^2, \qquad  \tfrac{1}{\sqrt{\upharpoonleft  \Omega_\varepsilon \! \upharpoonright_3}}\,
{\left\| \vec{V}_{o,\varepsilon} - \vec{\mathcal{V}}_{o,\varepsilon} \right\|}_{(L^2(\Omega^T_\varepsilon)^3}
 \le  \tilde{C}_2\, \varepsilon^2,
\end{equation}
\begin{equation}
\tfrac{1}{\sqrt{\upharpoonleft  \Omega_\varepsilon \! \upharpoonright_3}}\,
{\|P_{w, \varepsilon} - \mathfrak{P}_{w,\varepsilon}\|}_{L^2(0,T; H^1(\Omega_\varepsilon))} \le \tilde{C}_3\, \varepsilon^2,
\qquad
\tfrac{1}{\sqrt{\upharpoonleft  \Omega_\varepsilon \! \upharpoonright_3}}\,
{\|P_{o, \varepsilon} - \mathfrak{P}_{o,\varepsilon}\|}_{L^2(0,T; H^1(\Omega_\varepsilon))}
 \le  \tilde{C}_3\, \varepsilon^2. \label{app-estimate-pressures}
\end{equation}
\end{lemma}
\begin{proof} {\bf 1.} According to \eqref{differences-U-W} we have
\begin{multline}\label{AAA}
  \overrightarrow{{V}_\varepsilon} = - \lambda(S_\varepsilon) \, \mathbb{K}_\varepsilon(x)  \nabla P_\varepsilon  = - \lambda(s_0 - W_\varepsilon) \, \mathbb{K}_\varepsilon(x)  \nabla(\mathfrak{P}_\varepsilon - U_\varepsilon)
   \\
  =    - \lambda(s_0) \, \mathbb{K}_\varepsilon \nabla\mathfrak{P}_\varepsilon + \lambda(s_0) \, \mathbb{K}_\varepsilon  \nabla  U_\varepsilon +  \lambda'(\Theta^{(1)}_\varepsilon)\, W_\varepsilon \, \mathbb{K}_\varepsilon  \nabla\mathfrak{P}_\varepsilon -
 \lambda'(\Theta^{(1)}_\varepsilon)\, W_\varepsilon \, \mathbb{K}_\varepsilon  \nabla U_\varepsilon
 \\
  = \overrightarrow{\mathcal{V}_\varepsilon} - \varepsilon^2  \left(\begin{matrix}
                                                \lambda(s_0) \, k_1 \, \partial_{x_1}u_2
                                                 \\ 0 \\ 0
                                              \end{matrix}\right)
 + \lambda(s_0) \, \mathbb{K}_\varepsilon  \nabla  U_\varepsilon
  \\
  +  \lambda'(\Theta^{(1)}_\varepsilon)\, W_\varepsilon \, \mathbb{K}_\varepsilon  \nabla\mathfrak{P}_\varepsilon -
 \lambda'(\Theta^{(1)}_\varepsilon)\, W_\varepsilon \, \mathbb{K}_\varepsilon  \nabla U_\varepsilon.
 \end{multline}
 Now, using estimates \eqref{app-estimate1} and \eqref{app-estimate2} and taking into account \eqref{saturation1} and \eqref{limit-saturation}, we derive  estimate \eqref{app-estimate3} from \eqref{AAA}.

 In the same way, we deduce
 \begin{align}\label{rep-2}
  \vec{V}_{w,\varepsilon} &= - \Lambda(S_\varepsilon) \, \mathbb{K}_\varepsilon(x)  \nabla S_\varepsilon  + b(S_\varepsilon) \, \overrightarrow{{V}_\varepsilon}
  \\
   & = \Lambda(s_0 - W_\varepsilon) \, \mathbb{K}_\varepsilon(x)  \nabla(s_0 - W_\varepsilon)  + b(s_0 - W_\varepsilon) \, \overrightarrow{{V}_\varepsilon}\notag
   \\
   & = \vec{\mathcal{V}}_{w,\varepsilon} + \Lambda(s_0 ) \, \mathbb{K}_\varepsilon  \nabla W_\varepsilon
   +  \Lambda'(\Theta^{(2)}_\varepsilon)\, W_\varepsilon \, \mathbb{K}_\varepsilon\nabla s_0
   - \Lambda'(\Theta^{(2)}_\varepsilon)\, W_\varepsilon \, \mathbb{K}_\varepsilon\nabla W_\varepsilon \notag
   \\& \ \ \ + b(s_0) \big(\overrightarrow{{V}_\varepsilon} -   \overrightarrow{\mathcal{V}_\varepsilon}\big)
   - b'(\Theta^{(3)}_\varepsilon) \, W_\varepsilon \, \overrightarrow{\mathcal{V}_\varepsilon}
      - b'(\Theta^{(3)}_\varepsilon) \, W_\varepsilon \, \big(\overrightarrow{{V}_\varepsilon} -   \overrightarrow{\mathcal{V}_\varepsilon}\big).
      \notag
  \end{align}
  Using the same inequalities as before and inequality \eqref{app-estimate3}, we get the first estimate in \eqref{app-estimate4} from equality \eqref{rep-2}. The second estimate in \eqref{app-estimate4} follows from the previous estimates, from \eqref{app-velocity-o} and from the equality $\vec{V}_{o, \varepsilon} = \overrightarrow{V}_\varepsilon - \vec{V}_{w,\varepsilon}.$

 {\bf 2.} From \eqref{mean Pressure} we deduce that
 \begin{equation}\label{diff-presure-w}
  P_{w, \varepsilon} - \mathfrak{P}_{w, \varepsilon} = P_{\varepsilon} - \mathfrak{P}_{\varepsilon} - \int_{s_0(x_1,t)}^{S_\varepsilon(x,t)} f(\eta)\, d\eta,
 \end{equation}
where $f(\eta):= \frac{\lambda_o(\eta)}{\lambda(\eta)} \, p_c^\prime(\eta).$ Taking into account the assumptions for the relative permeabilities and the capillary pressure (see \eqref{Leveret} - \eqref{total mobillity}), $f$ and its derivative are smooth and bounded in the domain of definition.

It is easy to see that
\begin{equation}\label{inequa1}
  \int_{\Omega^T_\varepsilon} \bigg(\int_{s_0(x_1,t)}^{S_\varepsilon(x,t)} f(\eta)\, d\eta\bigg)^2 dxdt \le C_1  \int_{\Omega^T_\varepsilon} \big(S_\varepsilon(x,t) - s_0(x_1,t)\big)^2 dxdt,
\end{equation}
\begin{equation}\label{inequa2}
  \nabla_{\overline{x}_1}\bigg(\int_{s_0(x_1,t)}^{S_\varepsilon(x,t)} f(\eta)\, d\eta\bigg) = f(S_\varepsilon)\, \nabla_{\overline{x}_1}S_\varepsilon(x,t),
\end{equation}
\begin{align}\label{inequa3}
  \partial_{x_1}\bigg(\int_{s_0(x_1,t)}^{S_\varepsilon(x,t)} f(\eta)\, d\eta\bigg)& = f(S_\varepsilon)\, \partial_{x_1}S_\varepsilon(x,t) -
  f(s_0)\, \partial_{x_1}s_0(x_1,t)\notag
  \\
 & = f(S_\varepsilon)\, \big(\partial_{x_1}S_\varepsilon - \partial_{x_1}s_0\big) + \partial_{x_1}s_0 \, f'(\Theta_\varepsilon)\, \big(S_\varepsilon - s_0\big).
\end{align}

Considering \eqref{inequa1} - \eqref{inequa3} and using \eqref{app-estimate1} and \eqref{app-estimate2}, we derive  the first estimate in \eqref{app-estimate-pressures} from \eqref{diff-presure-w}. The second one follows from \eqref{capilarity},  the equality
$$
P_{o, \varepsilon} - \mathfrak{P}_{o, \varepsilon} = P_{w, \varepsilon} - \mathfrak{P}_{w, \varepsilon} + p'_c(\Theta_\varepsilon^{(1)}) \, \big(S_\varepsilon - s_0\big)
$$
and the boundedness of $p'_c.$
\end{proof}

\smallskip

$\blacklozenge$ If  $\alpha > 1$ and $\beta =0,$ the approximations of
the velocities are as follows
\begin{equation}\label{app-total-velocity-al}
  \overrightarrow{\mathcal{V}_\varepsilon} := -  \left(\begin{matrix}
                                                \lambda(s_0) \, k_1 \, \partial_{x_1}p_0 + \varepsilon^{\alpha -1}\big(\lambda(s_0) \, k_1 \, \partial_{x_1}p_{\alpha -1} +
                                                \lambda'(s_0) \, k_1 \, \partial_{x_1}p_0 \, s_{\alpha -1}\big)
                                                 \\[4pt]
                                                \varepsilon^\alpha \, \lambda(s_0(x_1,t)) \, \mathbf{K}(x_1, \frac{\overline{x}_1}{\varepsilon}) \, \big(\nabla_{\overline{\xi}_1} u_{\alpha +1}(x_1,\overline{\xi}_1)\big)|_{\overline{\xi}_1 = \frac{\overline{x}_1}{\varepsilon}}
                                              \end{matrix}\right),
\end{equation}
\begin{multline}\label{app-velocity-w-al}
  \vec{\mathcal{V}}_{w,\varepsilon} := -\left(\begin{matrix}
                                                \Lambda(s_0) \, k_1 \, \partial_{x_1}s_0
                                                + \varepsilon^{\alpha -1}\big(\Lambda(s_0) \, k_1 \, \partial_{x_1}s_{\alpha -1} +
                                                \Lambda'(s_0) \, k_1 \, \partial_{x_1}s_0 \, s_{\alpha -1}\big)
                                                 \\ 0 \\ 0
                                              \end{matrix}\right)
        \\
   + \big(b(s_0)  + \varepsilon^{\alpha -1} b'(s_0) \, s_{\alpha-1}\big) \, \overrightarrow{\mathcal{V}_\varepsilon},
\end{multline}
where $(p_0, s_0)$ and $(p_{\alpha-1}, s_{\alpha-1})$ are  solutions to problems \eqref{limit_prob-p0-s0-el} and \eqref{limit_prob-al-be}, respectively, $u_{\alpha+1}$ is a solution to problem \eqref{Neumann u2}, and for the pressure ${P}_{w, \varepsilon}$ is
\begin{equation}\label{app-pressure-w-al}
  \mathfrak{P}_{w, \varepsilon} := \mathfrak{P}_{\varepsilon} - \int_{0}^{\mathfrak{S}_{\varepsilon}} \frac{\lambda_o(\eta)}{\lambda(\eta)} \, p_c^\prime(\eta) \,d\eta,
\end{equation}
where $\mathfrak{P}_{\varepsilon}$ and $\mathfrak{S}_{\varepsilon}$ are now determined in \eqref{Anz-P-al} and \eqref{Anz-S-al}, respectively.

 Now, using the approximation formulas \eqref{app-estimate1-al} and \eqref{app-estimate2-al} similarly to the previous one, we prove the lemma.
\begin{lemma}[$\alpha >1,$ $\beta =0$]\label{cor-4-2}
Let $\overrightarrow{\mathcal{V}_\varepsilon}$ and $\vec{\mathcal{V}}_{w,\varepsilon}$  be vector-valued functions determined by formulas \eqref{app-total-velocity-al} and \eqref{app-velocity-w-al} respectively,
and  $ \mathfrak{P}_{w, \varepsilon}$ be a function defined by \eqref{app-pressure-w-al}.
Then the following asymptotic estimates hold:
  \begin{equation}\label{app-estimate5}
\tfrac{1}{\sqrt{\upharpoonleft  \Omega_\varepsilon \! \upharpoonright_3}}\,  \max_{t\in \times [0,T]}
{\left\| \overrightarrow{{V}_\varepsilon}(\cdot,t) - \overrightarrow{\mathcal{V}_\varepsilon}(\cdot,t)\right\|}_{(L^2(\Omega_\varepsilon))^3}
 + \tfrac{1}{\sqrt{\upharpoonleft  \Omega_\varepsilon \! \upharpoonright_3}}\,
{\left\| \vec{V}_{w,\varepsilon} - \vec{\mathcal{V}}_{w,\varepsilon} \right\|}_{(L^2(\Omega^T_\varepsilon)^3}
 \le \tilde{C}_0 \,  \varepsilon^{\aleph_\alpha},
\end{equation}
\begin{equation}\label{app-estimate6}
\tfrac{1}{\sqrt{\upharpoonleft  \Omega_\varepsilon \! \upharpoonright_3}}\,
{\|P_{w, \varepsilon} - \mathfrak{P}_{w,\varepsilon}\|}_{L^2(0,T; H^1(\Omega_\varepsilon))} \le  \tilde{C}_1\, \varepsilon^{\aleph_\alpha},
\end{equation}
where $\aleph_\alpha := \min\{2(\alpha-1), \, \alpha +1\}.$
\end{lemma}
\begin{remark}\label{Rem-6-1}
An  approximation for  $\vec{V}_{o,\varepsilon}$ is determined by formula \eqref{app-velocity-o} and
for ${P}_{o, \varepsilon}$ by the formula $\mathfrak{P}_{o, \varepsilon} := \mathfrak{P}_{w, \varepsilon} + p_c(\mathfrak{S}_{\varepsilon}).$
The corresponding asymptotic estimates follow from these formulas and estimates \eqref{app-estimate5}, \eqref{app-estimate6} and \eqref{app-estimate2-al}. They have the same order. Therefore, we do not write them out.
\end{remark}

\smallskip

$\blacklozenge$ If  $\alpha =1,$ $\beta < 2$ and $\beta \neq  0,$ the approximation for  the total velocity $\overrightarrow{V_\varepsilon}$
is given by the vector-function $\overrightarrow{\mathcal{V}_\varepsilon} = \big(\lambda(s_0) \, k_1 \, \partial_{x_1}p_0, 0, 0\big)^\top$ $(\top$ is the sign of transposition), and the approximations for the  velocity $\vec{V}_{w,\varepsilon}$ and the pressure $P_{w,\varepsilon}$ are defined by formulas \eqref{app-velocity-w} and \eqref{app-pressure-w}, respectively, where $\mathfrak{P}_\varepsilon$ is now determined by
\eqref{Anz-P-be}. For the same reasons as in Remark~\ref{Rem-6-1}, we do not mention the functions $\vec{V}_{o,\varepsilon}$ and $P_{o,\varepsilon}.$

\begin{lemma}[$\alpha =1,$ $\beta < 2$ and $\beta \neq  0$]\label{lemma-6-3}
 Suppose that the conditions of Theorem~\ref{Th_1+} are satisfied, and denote by $U_{w,\varepsilon}$ the  differences $P_{w,\varepsilon} - \mathfrak{P}_{w,\varepsilon}.$
Then the following asymptotic estimates hold:
\begin{equation}\label{app-estimateU-w-be}
\tfrac{1}{\sqrt{\upharpoonleft  \Omega_\varepsilon \! \upharpoonright_3}}\,
\sqrt{ \int_{0}^{T} \int_{\Omega_\varepsilon} |\partial_{x_1} U_{w,\varepsilon}|^2 \, dx dt
+ \varepsilon^\beta \int_{0}^{T} \int_{\Omega_\varepsilon} |\nabla_{\overline{x}_1}U_{w,\varepsilon}|^2 \, dx dt} \le
 \tilde{C}_1\, \varepsilon^{1 - \frac{\beta}{2}},
\end{equation}
  \begin{equation}\label{app-estimate7}
\tfrac{1}{\sqrt{\upharpoonleft  \Omega_\varepsilon \! \upharpoonright_3}}\,  \max_{t\in \times [0,T]}
{\| \overrightarrow{{V}_\varepsilon}(\cdot,t) - \overrightarrow{\mathcal{V}_\varepsilon}(\cdot,t) \|}_{(L^2(\Omega_\varepsilon))^3}
+ \tfrac{1}{\sqrt{\upharpoonleft  \Omega_\varepsilon \! \upharpoonright_3}}\,
{\| \vec{V}_{w,\varepsilon} - \vec{\mathcal{V}}_{w,\varepsilon} \|}_{(L^2(\Omega^T_\varepsilon)^3}
 \le
 \tilde{C}_2\, \varepsilon^{\aleph_\beta},
\end{equation}
where $\aleph_\beta := \min\{1 - \frac{\beta}{2}, \, 1\}.$
\end{lemma}
\begin{proof}
First, note that from \eqref{diff-presure-w} it follows that $U_{w,\varepsilon}\big|_{x_1=0}=U_{w,\varepsilon}\big|_{x_1=\ell}= 0.$ Then, similarly to the second part of the proof of Lemma~\ref{lemma-6-1}, but now using  estimates \eqref{app-estimate-U-e-be} - \eqref{app-estimateW-ep2-be}, we derive \eqref{app-estimateU-w-be}.

Next, as in the first part of the proof of Lemma~\ref{lemma-6-1}, we prove
\begin{align}\label{rep-3}
  \overrightarrow{{V}_\varepsilon} &= \overrightarrow{\mathcal{V}_\varepsilon} - \varepsilon^{2-\beta} \big(\lambda(s_0) \, k_1 \, \partial_{x_1}u_{2-\beta}, \, 0,\, 0\big)^\top - \varepsilon \big(0, 0, \lambda(s_0) \, \mathbf{K} \nabla_{\overline{\xi}_1}u_{2-\beta}\big)^\top
      \\
 & \ \ \ +  \lambda(s_0) \, \mathbb{K}_\varepsilon  \nabla  U_\varepsilon +  \lambda'(\Theta^{(1)}_\varepsilon)\, W_\varepsilon \, \mathbb{K}_\varepsilon  \nabla\mathfrak{P}_\varepsilon -
 \lambda'(\Theta^{(1)}_\varepsilon)\, W_\varepsilon \, \mathbb{K}_\varepsilon  \nabla U_\varepsilon. \notag
 \end{align}
 New in this proof is estimating the norm
\begin{multline}\label{estim-norm-be}
  \tfrac{1}{\sqrt{\upharpoonleft  \Omega_\varepsilon \! \upharpoonright_3}}\,  \max_{t\in \times [0,T]}
{\| \mathbb{K}_\varepsilon  \nabla  U_\varepsilon\|}_{(L^2(\Omega_\varepsilon))^3}
\\
\le
\tfrac{C_1}{\sqrt{\upharpoonleft  \Omega_\varepsilon \! \upharpoonright_3}}\,  \max_{t\in \times [0,T]} \Big(
{\| \partial_{x_1} U_\varepsilon\|}_{L^2(\Omega_\varepsilon)} + {\| \varepsilon^\beta \nabla_{\overline{x}_1}U_\varepsilon\|}_{(L^2(\Omega_\varepsilon))^2}\Big)
 \stackrel{\eqref{app-estimate-U-e-be}}{\le} C_2 \big(\varepsilon^{1 - \frac{\beta}{2}} + \varepsilon\big).
\end{multline}
Then, using  \eqref{app-estimate-U-e-be} - \eqref{app-estimateW-ep2-be},  we get \eqref{app-estimate7}.
\end{proof}

\smallskip

$\blacklozenge$   If  $\alpha > \beta -1, \,\alpha >1 $ and $\beta \neq 0,$  the approximation for the pressure ${P}_{w, \varepsilon}$ is determined by formula~\eqref{app-pressure-w-al}, but  $\mathfrak{P}_{\varepsilon}$ and $\mathfrak{S}_{\varepsilon}$ are now determined in \eqref{Anz-P-al-be} and \eqref{Anz-S-al-be}, respectively. The estimate for the difference $U_{w,\varepsilon} :=P_{w,\varepsilon} - \mathfrak{P}_{w,\varepsilon}$  is derived in the same way as in Lemmas~\ref{lemma-6-3} and \ref{lemma-6-1}, but it is necessary to use estimates~\eqref{app-estimate-U-e-al-be} - \eqref{app-estimateW-ep2-al-be}.

In comparison to the other cases, the order of the estimates is the smallest in this case. As follows from the previous case, the order
of estimates for the flow velocities will become even smaller, indicating a complex influence of the parameters $\alpha$ and $\beta$ on the course of the process. We now propose the following approximations:
\begin{equation}\label{app-velocities-al-be}
  \overrightarrow{\mathcal{V}_\varepsilon} := -  \big(\lambda(s_0) \, k_1 \, \partial_{x_1}p_0 , \, 0, \, 0\big)^\top,\quad
 \vec{\mathcal{V}}_{w,\varepsilon}:= -\big( \Lambda(s_0) \, k_1 \, \partial_{x_1}s_0, \, 0, \,  0 \big)^\top + \, b(s_0) \, \overrightarrow{\mathcal{V}_\varepsilon},
\end{equation}
where $(p_0, s_0)$ is a  solution to problem \eqref{limit_prob-p0-s0-el}.

Proceeding in the same way as before and using estimates~\eqref{app-estimate-U-e-al-be} - \eqref{app-estimateW-ep2-al-be}, we derive
\begin{equation}\label{deduce1}
  \overrightarrow{{V}_\varepsilon} = \overrightarrow{\mathcal{V}_\varepsilon} + \mathcal{O}(\varepsilon^{\alpha-1}) +
\mathcal{O}(\varepsilon^{\aleph_{\alpha,\beta}}) + \mathcal{O}(\varepsilon^{\frac{\beta}{2} + \aleph_{\alpha,\beta}}) \quad \text{as} \ \ \varepsilon \to 0,
\end{equation}
where $\aleph_{\alpha,\beta} = \min\{ \alpha-1 , \, \frac{\alpha-\beta +1}{2}\}.$ So, if $\alpha \ge 3 - \beta$ (this means $\aleph_{\alpha,\beta} = \frac{\alpha-\beta +1}{2}),$ then
\begin{equation*}
  \overrightarrow{{V}_\varepsilon} = \overrightarrow{\mathcal{V}_\varepsilon} +  \left\{
                                                                                   \begin{array}{ll}
                                                                                    \mathcal{O}\big(\varepsilon^{\frac{\alpha-\beta +1}{2}}\big) , & \hbox{when} \ \beta > 0; \\
                                                                                     \mathcal{O}\big(\varepsilon^{\frac{\alpha+1}{2}}\big), & \hbox{when} \ \beta< 0.
                                                                                   \end{array}
                                                                                 \right.
\end{equation*}
If $\alpha < 3 - \beta,$ then $ \overrightarrow{{V}_\varepsilon} = \overrightarrow{\mathcal{V}_\varepsilon} +  \mathcal{O}\big(\varepsilon^{\alpha- 1}\big)$ for $\beta >0$ (due to the other restrictions for $\alpha$ and $\beta$ this means that $\beta \in (0, 2)),$ and $ \overrightarrow{{V}_\varepsilon} = \overrightarrow{\mathcal{V}_\varepsilon} +  \mathcal{O}\big(\varepsilon^{\alpha - 1+ \frac{\beta}{2}}\big)$ for $\beta < 0$ and $\alpha > -\frac{\beta}{2} +1.$ If $1< \alpha <  -\frac{\beta}{2} +1,$ additional terms of the asymptotics for $P_\varepsilon$ and $S_\varepsilon$ should be constructed (see Remark~\ref{remark-4-1}).
In reality, the parameter $\beta$ is positive because the absolute permeability coefficients for different materials are very small ($10^{-3}$ - $10^{-16} \, \mathrm{cm}^2$, see e.g. \cite[\S 5.5]{Bear-1972}). Therefore, for this case, we formulate the results for positive $\beta.$

\begin{lemma}[$\alpha > \beta -1, \,\alpha >1 $ and $\beta \neq 0$]\label{lemma-6-4}
 Suppose that the conditions of Theorem~\ref{Th_5-1} are satisfied. Then the following asymptotic estimates hold:
\begin{equation*}%\label{app-estimateU-w-al-be}
\tfrac{1}{\sqrt{\upharpoonleft  \Omega_\varepsilon \! \upharpoonright_3}}\,
\sqrt{ \int_{0}^{T} \int_{\Omega_\varepsilon} |\partial_{x_1} U_{w,\varepsilon}|^2 \, dx dt
+ \varepsilon^\beta \int_{0}^{T} \int_{\Omega_\varepsilon} |\nabla_{\overline{x}_1}U_{w,\varepsilon}|^2 \, dx dt} \le
 \tilde{C}_1\, \varepsilon^{\aleph_{\alpha,\beta}},
\end{equation*}
where $U_{w,\varepsilon} = P_{w,\varepsilon} - \mathfrak{P}_{w,\varepsilon}$ and
$\aleph_{\alpha,\beta} = \min\{ \alpha-1 , \, \frac{\alpha-\beta +1}{2}\};$
  \begin{equation*}%\label{app-estimate8}
\tfrac{1}{\sqrt{\upharpoonleft  \Omega_\varepsilon \! \upharpoonright_3}}\,  \max_{t\in \times [0,T]}
{\| \overrightarrow{{V}_\varepsilon}(\cdot,t) - \overrightarrow{\mathcal{V}_\varepsilon}(\cdot,t) \|}_{(L^2(\Omega_\varepsilon))^3}
+ \tfrac{1}{\sqrt{\upharpoonleft  \Omega_\varepsilon \! \upharpoonright_3}}\,
{\| \vec{V}_{w,\varepsilon} - \vec{\mathcal{V}}_{w,\varepsilon} \|}_{(L^2(\Omega^T_\varepsilon)^3}
 \le
 \tilde{C}_2\,  \varepsilon^{\frac{\alpha-\beta +1}{2}}
\end{equation*}
for $\alpha \ge 3 - \beta$ and $\beta>0;$ and
  \begin{equation*}%\label{app-estimate9}
\tfrac{1}{\sqrt{\upharpoonleft  \Omega_\varepsilon \! \upharpoonright_3}}\,  \max_{t\in \times [0,T]}
{\| \overrightarrow{{V}_\varepsilon}(\cdot,t) - \overrightarrow{\mathcal{V}_\varepsilon}(\cdot,t) \|}_{(L^2(\Omega_\varepsilon))^3}
+ \tfrac{1}{\sqrt{\upharpoonleft  \Omega_\varepsilon \! \upharpoonright_3}}\,
{\| \vec{V}_{w,\varepsilon} - \vec{\mathcal{V}}_{w,\varepsilon} \|}_{(L^2(\Omega^T_\varepsilon)^3}
 \le
 \tilde{C}_3\,  \varepsilon^{\alpha-1}
\end{equation*}
for $\alpha <  3 - \beta$ and $\beta \in (0, 2).$
\end{lemma}

%%%%%%%%%%%%%%%%%%%%

%%%%%%%%%%%%%%
\section{Conclusions}\label{Sect: Concl}

{\bf 1.} In this work, we have examined the nonlinear Muskat-Leverett two-phase flow model in a thin cylinder~$\Omega_\varepsilon.$ The problem involves two additional parameters, $\alpha$ and $\beta$. The parameter $\alpha$ affects the flow rates of the mixture on the lateral surface $\Gamma_\varepsilon,$ while $\beta$ affects the absolute permeability tensor in the transverse directions.

The results demonstrate that for a thin cylinder with a cross-sectional diameter of order $\varepsilon,$ the value $\alpha = 1$ represents a critical threshold. This is observed in the first term of the asymptotics, and problem~\eqref{limit_prob} for $\beta < 2$ represents the corresponding one-dimensional Muskat-Leverett two-phase filtration model, taking into account the specified lateral surface flows from the original problem.
It can be reduced to the nonlinear parabolic equation \eqref{limit_parab-eq}.

For $\alpha >1$ and $\beta < 2,$  or $\alpha > \beta - 1$ and  $\beta \ge 2,$ the order of the specified flow rates of the mixture on the lateral surface $\Gamma_\varepsilon$ is relatively low. Consequently, its influence is observed only in the second terms of the asymptotics, namely $p_{\alpha-1}$ and $s_{\alpha-1},$ which form a solution to problem \eqref{limit_prob-al-be}. In these cases, the limit problem is problem \eqref{limit_prob-p0-s0-el}, and it is recommended that both this problem and the aforementioned problem to be considered in the context of potential applications.

Additionally, our results illustrate the correlation between these parameters: $\alpha \ge 1$ and $\beta < 2,$  or $\alpha > \beta - 1$ and  $\beta \ge 2.$ With regard to the parameter $\beta,$ the critical value is $2.$ From \eqref{Anz-P-al-be} and \eqref{Anz-S-al-be} we see that
the impact of a specific two-phase flow at the lateral boundary of the cylinder is particularly pronounced if $\alpha < 1$ and $\beta < 2,$
indicating that it may lead to substantial alterations in the overall process.
In this specific instance, the term $\varepsilon^{\alpha  -1} s_{\alpha  -1}$ of the asymptotics would become unbounded as $\varepsilon \to 0.$ However, this is not consistent with the physical interpretation of the model, which requires that the saturation remain within $[0,1].$
This indicates that there is a high level of conductivity through the lateral surface of the cylinder;  thus, this regime of fracture-rock matrix permeability relationships must be modeled using a different approach. The same for $\alpha < \beta - 1$ and  $\beta \ge 2.$

In the specific case $\alpha = \beta - 1$ and  $\beta \ge 2,$ ansatz \eqref{Anz-P-al-be} suggests
the advent of what is known as the "dual-porosity" regime in the thin cylinder $\Omega_\varepsilon;$
in our particular case, this signifies that the longitudinal two-phase flow exhibits a rate that is equivalent to that of the transverse flow.
%A periodic homogenization of the double-porosity model for immiscible incompressible two-phase flow was performed in \cite{Bou-Luck-Mik-1996}.

Knowledge of flow velocities is critical in many applications,  such as the oil and gas industry, where fractured reservoirs are common, or environmental engineering, where fractured media often play a critical role in the transport of contaminants. It enables informed and optimised decision making to ensure the safety and sustainability of fluid management.
Therefore, the accurate velocity approximations obtained for the velocities $\vec{V}_{w,\varepsilon}$ and $\vec{V}_{o,\varepsilon}$ in Section~\ref{Sect-6} are important results that can help to more effectively predict two-phase fluid motion.

\vspace{6pt}

{\bf 2.} This study is based on a number of assumptions, which we would like to comment on in more detail. The assumptions regarding the smoothness of the given functions $q_0,$ $q_\ell,$ $S_0,$ $Q$ and the coefficients of problem $(\mathbb{P}_\varepsilon\mathbb{S}_\varepsilon\!)$  is necessary to determine the terms of the asymptotic approximations and to estimate certain integrals in the proofs.
The assumption that the function $Q$ vanishes in small neighbourhoods of the bases of the the thin cylinder corresponds to many real situations to a greater extent. Indeed, the penetration of the mixture from a solid rock into a crack is a local phenomenon, which indicates that  $Q$ has a compact support. Furthermore, from a mathematical standpoint, this assumption enables the avoidance of constructing  boundary layers near the bases of the thin cylinder to satisfy the boundary conditions there.

The main assumptions of our work are the regularity of capillary pressure and the neglect of acceleration of gravity in Darcy's law. They ensure the maximum principle (see \eqref{saturation1}), which in turn guarantees the regularity of the problem in the sense of inequality \eqref{regular} and the possibility of proving  asymptotic estimates for solutions. It should be noted that a number of studies directly assume that the problem being studied is regular (see, e.g., \cite{Bou-Luck-Mik-1996,Kro-Luckhaus-1984,Kruz-Suk-1977}).

Therefore, it would be intriguing to obtain findings without making such assumptions.
Investigating the dual-porosity regime, particularly for the case  $\alpha = 1$ and $\beta = 2$, would also be an interesting and worthwhile endeavor. Nevertheless, our future paper will focus on a nonlinear Muskat-Leverett two-phase filtration problem in thin graph-like networks that more accurately model fracture geometry. The main challenge will be deriving the limit problem on a graph with appropriate Kirchhoff transmission conditions at the vertices.

\subsection*{Acknowledgment}
The authors  thank  for  funding by the  Deutsche Forschungsgemeinschaft (DFG, German Research Foundation) – Project Number 327154368 -- SFB 1313.

\bibliographystyle{siamplain}
\bibliography{references-Melnyk-Rohde}
\end{document}